\documentclass[final,leqno,onefignum]{siamltex704}

\usepackage{fancyhdr}
\usepackage{mathrsfs}
\usepackage{latexsym}
\usepackage{amsfonts}
\usepackage[usenames]{color}
\usepackage{graphicx}
\usepackage{amsmath}
\usepackage{enumitem}
\usepackage{epstopdf}

\newcommand{\be}[1]{\begin{equation} \label{#1}}
\newcommand{\ee}{\end{equation}}

\newcommand{\ii}{\mathrm{i}}

\begin{document}

\title{Inverting the local geodesic ray transform of higher rank tensors}
\author{Maarten V. de Hoop \thanks{Simons Chair in Computational and
    Applied Mathematics and Earth Science, Rice University, Houston,
    TX, 77005, USA (mdehoop@rice.edu)}
\and Gunther Uhlmann \thanks{Department of Mathematics, University of Washington,
Seattle, WA, 98195-4350, US, HKUST Jockey Club Institute for Advanced Study, HKUST, Clear Water Bay, Kowloon, Hong Kong, China (gunther@math.washington.edu)}
\and Jian Zhai \thanks{Department of Mathematics, University of Washington, Seattle, WA, 98195-4350, USA
  (jian.zhai@outlook.com)}}

\maketitle

\pagestyle{myheadings}
\thispagestyle{plain}
\markboth{DE HOOP, UHLMANN, and ZHAI}{Local geodesic ray transform of higher rank tensors}

\begin{abstract}
Consider a Riemannian manifold in dimension $n\geq 3$ with strictly
convex boundary. We prove the local invertibility, up to potential
fields, of the geodesic ray transform on tensor fields of rank four
near a boundary point. This problem is closely related with elastic
\textit{qP}-wave tomography. Under the condition that the manifold can
be foliated with a continuous family of strictly convex hypersurfaces,
the local invertibility implies a global result. One can
straightforwardedly adapt the proof to show similar results for tensor
fields of arbitrary rank.
\end{abstract}
\section{Introduction}
We let $M\subset\mathbb{R}^3$ be a bounded domain with smooth
boundary $\partial M$ and $x=(x^1,x^2,x^3)$ be the Cartesian coordinates. The system of equations describing elastic waves reads
\begin{equation}\label{EQ no1}
\rho\partial^2_tu=\operatorname{div} (\mathbf{C}\varepsilon(u)).\\
\end{equation}
Here, $u$ denotes the displacement vector and
\[\varepsilon(u)=(\nabla u+(\nabla u)^T)/2=(\varepsilon_{ij}(u))=\frac{1}{2}\left(\frac{\partial u_i}{\partial x^j}+\frac{\partial u_j}{\partial x^i}\right)\]
the linear strain tensor which is the symmetric part of $\nabla u$. Furthermore, $\mathbf{C}=(C_{ijkl})=(C_{ijkl}(x))$ is the stiffness
tensor and $\rho=\rho(x)$ is the density of mass. 

The stiffness tensor is assumed to have the symmetries
\[C_{ijkl}=C_{jikl}=C_{klij}.\]
 The operator $\operatorname{div} (\mathbf{C}\varepsilon(\cdot))$ is elliptic if we additionally assume that there exists a $\delta>0$ such that for any 
$3\times 3$ real-valued symmetric matrix $(\varepsilon_{ij})$,
\[
   \sum_{i,j,k,l=1}^3C_{ijkl}\varepsilon_{ij}\varepsilon_{kl}\geq\delta\sum_{i,j=1}^3\varepsilon_{ij}^2.
\]

If the stiffness tensor $\mathbf{C}$ is isotropic, we have
\begin{equation}\label{symmetry}
C_{ijkl}=\lambda\delta_{ij}\delta_{kl}+\mu(\delta_{ik}\delta_{jl}+\delta_{il}\delta_{jk}),
\end{equation}
where $\lambda,\mu$ are  called the Lam\'{e} parameters. For isotropic elasticity there are two different wave-speeds, namely, \textit{P}-wave (longitudinal wave) speed $c_P=\sqrt{\frac{\lambda+2\mu}{\rho}}$ and S-wave  (transverse wave) speed $c_S=\sqrt{\frac{\mu}{\rho}}$. Then we can consider $M$ as a manifold with metric $c_P^{-2}\mathrm{d}s^2$ or $c_S^{-2}\mathrm{d}s^2$. Correspondingly, we can view \textit{P} waves traveling along geodesics in Riemannian manifold $(M,c_P^{-2}\mathrm{d}s^2)$, and \textit{S} waves traveling along geodesics in $(M,c_S^{-2}\mathrm{d}s^2)$.

If there is an anisotropic perturbation $a_{ijkl}$ around isotropy, that is,
\[
C_{ijkl}=\lambda\delta_{ij}\delta_{kl}+\mu(\delta_{ik}\delta_{jl}+\delta_{il}\delta_{jk})+a_{ijkl},
\]
the perturbation in travel time of \textit{P}-waves along a geodesic $\gamma$ gives the following quantity \cite{CJ}:
\begin{equation}\label{btensor}
\int_{\gamma}\frac{a_{ijkl}}{\rho c_P^6}\dot{\gamma}^i\dot{\gamma}^j\dot{\gamma}^k\dot{\gamma}^l\mathrm{d}t.
\end{equation}
Here $\gamma$ is a geodesic in $(M,c_P^{-2}\mathrm{d}s^2)$. The same
quantity has been derived by a different perturbation analysis
\cite{Shara}. Equation (\ref{btensor}) represents a geodesic ray
transform of a 4-tensor $b_{ijkl}=\frac{a_{ijkl}}{\rho c_P^6}$ in
$(M,c_P^{-2}\mathrm{d}s^2)$.

Let $(M,g)$ be a compact Riemannian manifold with boundary $\partial
M$. The geodesic ray transform of a symmetric tensor field $f$ of
order $m$ is given by
\begin{equation}\label{Im}
I_mf(\gamma)=\int_{\gamma}\langle f(\gamma(t)),\dot{\gamma}^m(t)\rangle\mathrm{d}t,
\end{equation}
where, in local coordinates, $\langle
f,v^m(t)\rangle=f_{i_1,\cdots,i_m}v^{i_1}\cdots v^{i_m}$, and $\gamma$
runs over all geodesics with endpoints on $\partial M$. We note, here,
that the tensor $b$ in $(\ref{btensor})$ is not fully symmetric. Thus,
we introduce $f$ that is the symmetrization of $b$, and study the
geodesic X-ray transform $I_4 f$. A general tensor with symmetry
(\ref{symmetry}) has 21 unknowns, while a symmetric 4-tensor has 15
unknowns. Therefore we have already lost 6 components of $\mathbf{C}$
in the formulation of the problem.

It is known that potential vector fields, i.e., $f=\mathrm{d}^sv$ with
$v$ a symmetric field of order $m-1$ vanishing on $\partial M$ ($m\geq
1$), are in the kernel of $I_m$. Here, $\mathrm{d}^s$ is the symmetric
part of the covariant derivative $\nabla$, which will be defined in
(\ref{ds}). We say that $I_m$ is s-injective if $I_mf=0$ implies
$f=\mathrm{d}^sv$ with $v\vert_{\partial M}=0$. The s-injectivity of
$I_m$ has been extensively investigated, and we refer to
\cite{IM,SUVZ} for detailed reviews.

Assuming that $M$ is simple, when $\partial M$ is strictly
convex and any two points in $M$ are connected by a unique minimizing
geodesic smoothly depending on the endpoints, it has been proved that
$I_0$ is injective \cite{Muk1,Muk2}, and $I_1$ is s-injective
\cite{AR}. In dimension two, the s-injectivity of $I_m$ for arbitrary
$m$ is proved in \cite{PSU}. In dimension three or higher, the
s-injectivity of $I_m, m\geq 2$ is still open. When $(M,g)$ has
negative sectional curvature \cite{PS}, or under certain other
curvature conditions \cite{Dair, Pestov, Shara}, the s-injectivity has
been established. Without any curvature condition, it has been proved
that the problem is Fredholm \cite{SU} (modulo potential fields) with
a finite-dimensional smooth kernel. For analytic simple metrics, the
uniqueness is proved using microlocal analytic continuation. With the
Fredholm property, the uniqueness can be extended to an open and dense
set of simple metrics in $C^k,k\gg 1$, containing analytic simple
metrics.

In \cite{UV}, Uhlmann and Vasy proved that, if $\partial M$ is
strictly convex at $p\in\partial M$ in dimension three or higher,
$I_0f(\gamma)$, for all geodesics localized in some suitable $\Omega$
near $p$, determine $f$ near $p$. Furthermore, under some global
convex ``foliation condition", it gives a global result via layer
stripping techniques. Then, Stefanov, Uhlmann and Vasy gave
corresponding results for $I_1$ and $I_2$ \cite{SUV2}. The key point
is to show the ellipticity (under a suitable gauge condition) of a
different version of the normal operator $I_m^*I_m$ as a scattering
pseudodifferential operator. The calculation for $I_1$, $I_2$, which
is already massive, is not observed to have an easy extension to
$I_m$, $m \geq 3$. In this paper, we will prove parallel results for
$I_4$ for two main reasons: (1) it arises naturally from elastic
\textit{qP}-wave tomography; (2) the scheme of calculation needs to be general
enough so that one can easily adapt the procedure to prove similar
results for $I_m$ with arbitrary $m$.

For an open set $O\subset M$, $O\cap\partial M\neq \emptyset$, we call
$\gamma$ an $O$-local geodesic if $\gamma$ is a geodesic contained in
$O$ with endpoints in $\partial M$. We denote the set of $O$-local
geodesics by $\mathcal{M}_O$. Note that $\mathcal{M}_O$ is an open
subset of the set of all geodesics $\mathcal{M}$. The introduction of
$\mathcal{M}$ and $\mathcal{M}_O$ can be found in \cite{UV}.  We
define the local geodesic ray transform of $f$ as the collection $(I_m
f)(\gamma)$ along all geodesics $\gamma\in\mathcal{M}_O$, that is, as
the restriction of the geodesic ray transform to $\mathcal{M}_O$. We
will restrict ourselves to the problem (\ref{Im}) with $m=4$ from now
on.

First, we consider $M$ as a strictly convex domain in a Riemannian
manifold $(\tilde{M},g)$ (without boundary), with boundary defining
function $\rho$, such that $\rho\geq 0$ on $M$.  As in \cite{UV,
  SUV2}, we first study the invertibility of $I_4$ in a neighborhood
of a point $p\in\partial M$ of the form $\{\tilde{x}>-c\}$,
$c>0$. Here $\tilde{x}$ is a function with $\tilde{x}(p)=0$,
$\mathrm{d}\tilde{x}(p)=-\mathrm{d}\rho(p)$. We denote
$\Omega=\Omega_c=\{x\geq 0,\rho\geq 0\}$, $x=x_c=\tilde{x}+c$. Using
the local geodesic ray transform with $\Omega$-local geodesics, we
have the local injectivity result

\begin{theorem}
With $\Omega=\Omega_c$ as above, there is $c_0>0$ such that for $c\in
(0,c_0)$, if $f\in L^2(\Omega)$ is a symmetric $4$-tensor. then
$f=u+\mathrm{d}^sv$, where $v\in
\dot{H}_{loc}^1(\Omega\setminus\{x=0\})$, while $u\in
L^2_{loc}(\Omega\setminus\{x=0\})$ can be stably determined from
$I_4f$ restricted to $\Omega$-local geodesics in the following
sense. There is a continuous map $I_4f\mapsto u$, where for $s\geq 0$,
$f\in H^s(\Omega)$, the $H^{s-1}$ norm of $u$ restricted to any
compact subset of $\Omega\setminus\{x=0\}$ is controlled by the $H^s$
norm of $I_4f$ restricted to the set of $\Omega$-local geodesics.

Replacing $\Omega_c=\{\tilde{x}>-c\}\cap M$ by
$\Omega_{\tau,c}=\{\tau>\tilde{x}>-c+\tau\}\cap M$, $c$ can be taken
uniform in $\tau$ for $\tau$ in a compact set on which the strict
concavity assumption on level sets of $\tilde{x}$ holds.
\end{theorem}

The Sobolev spaces $\dot{H}_{loc}^1$ will be defined in Section
\ref{mainproof}. As in \cite{SUV2, UV}, the above theorem can be
applied to obtain the following global result. Now, assume $\tilde{x}$
is a globally defined function with level sets
$\Sigma_t=\{\tilde{x}=t\}$ strictly concave (viewed from
$\tilde{x}^{-1}(0,t)$) for $t\in(-T,0]$, with $\tilde{x}\leq 0$ on the
  manifold $M$ with boundary. Assume further that $\Sigma_0=\partial
  M$ and $M\setminus\cup_{t\in(-T,0]}\Sigma_t$ has measure $0$ or has
    an empty interior. Will say such an $M$ satisfies the foliation
    condition.

\begin{theorem}\label{mainth}
Suppose $M$ is compact. The geodesic ray transform is injective and
stable modulo potentials on the restriction of symmetric $4$-tensors
$f$ to $\tilde{x}^{-1}((-T,0])$ in the following sense. For all
  $\tau>-T$ there is $v\in\dot{H}^{1}_{loc}(\tilde{x}^{-1}((\tau,0]))$
    such that $f-\mathrm{d}^sv\in L^2_{loc}(\tilde{x}^{-1}((\tau,0]))$
      can be stably recovered from $I_4f$. Here for stability we
      assume that $s\geq 0$, $f$ is in an $H^s$-space, the norm on
      $I_4f$ is an $H^s$-norm, while the norm for $v$ is an
      $H^{s-1}$-norm.
\end{theorem}

The foliation condition can be satisfied even in the presence of
caustics. A Riemannian manifold $(M,c^{-2}(|x|)\mathrm{d}s^2)$
satisfying the Herglotz \cite{Her} and Wiechert and Zoeppritz
\cite{WZ} condition $\frac{\mathrm{d}}{\mathrm{d}r}\frac{r}{c(r)}>0$
satisfies the foliation condition. The Euclidean spheres $|x|=r$ form
a strictly convex foliation. With the PREM (Preliminary Reference
Earth Model) model for Earth, this condition is a realistic one. We
note here that it does not exclude the existence of conjugate
points. More discussion on the foliation condition can be found in
\cite{PSUZ} and the references therein.

\section{Pseudodifferential property}
In $\Omega$, we can use local coordinates $(x,y)$, with $x$ introduced above. We are interested in geodesics ``almost tangent" to level sets of $\tilde{x}$. 

Let $\gamma_{x,y,\lambda,\omega}$ be a geodesic in $\tilde{M}$ such that
\[\gamma_{x,y,\lambda,\omega}(0)=(x,y),~~\dot{\gamma}_{x,y,\lambda,\omega}(0)=(\lambda,\omega),\]
with
$(x,y,\lambda,\omega)\in\mathbb{R}\times\mathbb{R}^{n-1}\times\mathbb{R}\times\mathbb{S}^{n-2}$. We
need that for $x\geq 0$ and $\lambda$ sufficiently small the
geodesic $\gamma_{x,y,\lambda,\omega}(t)$ stays in $x\geq 0$ as long
as it is in $M$. Thus for $x=0$, $\lambda$ can only be $0$. This is
guaranteed if $|\lambda|<C_1\sqrt{x}$, for sufficiently small
$C_1$. For convenience, we use a smaller range $|\lambda|\leq
C_2x$. We take $\chi$ to be a smooth, even, non-negative function with
compact support (to be specified).

We denote
\begin{equation}
(I_4f)(x,y,\lambda,\omega)=\int_{\mathbb{R}}\langle f(\gamma_{x,y,\lambda,\omega}(t)),\dot{\gamma}_{x,y,\lambda,\omega}^4(t)\rangle\mathrm{d}t.
\end{equation}
We note here that we are only interested in $f$ supported in $\overline{M}$, whence the above integration is actually along the segment of $\gamma_{x,y,\lambda,\omega}$ in $M$.
On $u(x,y,\lambda,\omega)$, we define
\begin{equation}
\begin{split}
(L_4u)(x,y)=x^4\int\chi(\lambda/x)u(x,y,\lambda,\omega)&g_{sc}(\lambda\partial_x+\omega\partial_y)\otimes g_{sc}(\lambda\partial_x+\omega\partial_y)\\
&\otimes g_{sc}(\lambda\partial_x+\omega\partial_y)\otimes g_{sc}(\lambda\partial_x+\omega\partial_y)\mathrm{d}\lambda\mathrm{d}\omega.
\end{split}
\end{equation}
We will carry out the calculation on $X=\{x\geq 0\}$.
Here, $u$ is a (locally defined in the support of $\chi$) function on
the space of geodesics parametrized by $(x,y,\lambda,\omega)$,
and $g_{sc}$ maps vectors to covectors; $g_{sc}$ is the scattering
metric of the form
\begin{equation}
g_{sc}=x^{-4}\mathrm{d}x^2+x^{-2}h,
\end{equation}
where $h(x,y)$ is a standard 2-cotensor on $X$.

As in \cite{SUV2}, we will show that $L_4I_4$, conjugated by an
exponential weight, is in Melrose's scattering pseudodifferential
algebra (cf. \cite{Melrose} for an introduction). The ellipticity of
the scattering pseudodifferential operator will be the main subject of
this section. In local coordinates $(x,y^1,\cdots,y^{n-1})$, the
scattering tangent bundle ${}^{sc}TX$, has a local basis
$x\partial_x,x\partial_{y^1},\cdots,x\partial_{y^{n-1}}$, and the dual
bundle ${}^{sc}T^*X$ correspondingly has a local basis
$\frac{\mathrm{d}x}{x^2},\frac{\mathrm{d}y^1}{x},\cdots,
\frac{\mathrm{d}y^{n-1}}{x}$. We adopt the notation
$\Psi_{sc}^{m,l}(X)$ for the scattering pseudodifferential algebra
introduced in \cite{SUV2}. We also use the notation ${}^{sc}TX$,
${}^{sc}T^*X$ and $\mathrm{Sym}^{k}{}^{sc}T^*X$ defined there in the
following analogue of \cite[Proposition 3.1]{SUV2}

\begin{proposition}
On symmetric $4$-tensors, the operator $N_\mathsf{F}=e^{-\mathsf{F}/x}LI_4e^{\mathsf{F}/x}$, lies in
\[\Psi_{sc}^{-1,0}(X;\mathrm{Sym}^{4}{}^{sc}T^*X,\mathrm{Sym}^{4}{}^{sc}T^*X),\]
for $\mathsf{F}>0$.
\end{proposition}
\begin{proof}
This proposition is analogous to . Use the map introduced in \cite{UV},
\[
\Gamma_+:S\tilde{M}\times[0,\infty]\rightarrow[\tilde{M}\times\tilde{M};\text{diag}],~\Gamma_+(x,y,\lambda,\omega,t)=(x,y,|y'-y|,\frac{x'-x}{|y'-y|},\frac{y'-y}{|y'-y|}),
\]
where $(x',y')=\gamma_{x,y,\lambda,\omega}(t)$. Here $[\tilde{M}\times\tilde{M};\text{diag}]$ is the \textit{blow-up} of $\tilde{M}$ at the diagonal $(x,y)=(x',y')$. Similarly, we can also define $\Gamma_-$ in which $(-\infty,0]$ takes the place of $[0,\infty)$.

We write
\[(\gamma_{x,y,\lambda,\omega}(t),\dot{\gamma}_{x,y,\lambda,\omega}(t))=(\mathsf{X}_{x,y,\lambda,\omega}(t),\mathsf{Y}_{x,y,\lambda,\omega}(t),\Lambda_{x,y,\lambda,\omega}(t),\Omega_{x,y,\lambda,\omega}(t)),\]
in coordinates $(x,y,\lambda,\omega)$ for lifted geodesic $\gamma_{x,y,\lambda,\omega}(t)$. We use the coordinates,
\[x,y,X=\frac{x'-x}{x^2},Y=\frac{y'-y}{x},\]
as in \cite{UV}, and obtain the Schwartz kernel of $N_\mathsf{F}$ on symmetric 4-tensors (with $\hat{Y}=\frac{Y}{|Y|}$):
\begin{equation}\label{kernel}
\begin{split}
K^\flat(x,y,X,Y)=\sum_{\pm}e^{-\mathsf{FX}/(1+xX)}\chi\left(\frac{X-\alpha(x,y,x|Y|,\frac{xX}{|Y|},\hat{Y})|Y|^2}{|Y|}+x\tilde{\Lambda}_{\pm}\left(x,y,x|Y|,\frac{x|X|}{|Y|},\hat{Y}\right)\right)\\
\left[x^{-1}(\Lambda\circ\Gamma_{\pm}^{-1})\frac{\mathrm{d}x}{x^2}+(\Omega\circ\Gamma_{\pm}^{-1})\frac{h(\partial_y)}{x}\right]^4\left[x^{-1}(\Lambda'\circ\Gamma_{\pm}^{-1})x^2\partial_{x'}+(\Omega'\circ\Gamma_{\pm}^{-1})x\partial_{y'}\right]^4\\
|Y|^{-n+1}J_{\pm}\left(x,y,\frac{X}{|Y|},|Y|,\hat{Y}\right).
\end{split}
\end{equation}
\end{proof}

We denote 
\[\nabla:T^mM\rightarrow T^{m+1}M\]
being the connection defined componentwise as 
\begin{equation}\label{connection}
\begin{split}
\nabla_k u_{j_1,\cdots,j_m}&=u_{j_1,\cdots,j_m;k}\\
&=\frac{\partial}{\partial x^k}u_{j_1,\cdots,j_m}-\sum_{p=1}^m\Gamma_{k,j_p}^qu_{j_1,\cdots,j_{p-1},q,j_{p+1},\cdots,j_m},
\end{split}
\end{equation}
where $\Gamma$ is the Christoffel symbol with respect to the metric $g$.
For $u\in T^mM$, we define its symmetrization as
\[\begin{split}\mathscr{S}: T^mM&\rightarrow S^mM\\
u&\mapsto f,
\end{split}\]
with
\[f(v_1,\cdots,v_m)=\frac{1}{m!}\sum_{\sigma}u(v_{\sigma(1)},\cdots,v_{\sigma(m)}),\]
where $\sigma$ runs over all permutation group of $(1,\cdots,m)$, and $v_j\in C^\infty(TM)$, $j=1,\cdots,m$.

We define the symmetric differential $\mathrm{d}^s\in S^mM\rightarrow S^{m+1}M$ to be
\begin{equation}\label{ds}
\mathrm{d}^s=\mathscr{S}\nabla.
\end{equation}
%Then, if $f$ is already a symmetric m-cotensor, then
%\[(\mathrm{d}^sf)_{j_1,\cdots,j_m,k}=\frac{1}{2}(u_{j_1,\cdots,j_m;k}+u_{j_1,\cdots,j_{m-1},k;j_{m}})\]
and note that $\mathrm{d}^s$ is different from the exterior
differential $\mathrm{d}$ defined on the bundle of $k$-forms
$\Lambda^kM$.  We also define
$\mathrm{d}_\mathsf{F}^s=e^{-\mathsf{F}/x}\mathrm{d}^se^{\mathsf{F}/x}$ and denote
its adjoint with respect to the scattering
metric $g_{sc}$ (not $g$) as $\delta_\mathsf{F}^s$.

For convenience of calculation, we will use the basis 
\[
\begin{split}
&\frac{\mathrm{d}x}{x^2}\otimes \frac{\mathrm{d}x}{x^2}\otimes \frac{\mathrm{d}x}{x^2},\\
&\frac{\mathrm{d}x}{x^2}\otimes \frac{\mathrm{d}x}{x^2}\otimes \frac{\mathrm{d}y}{x},~\frac{\mathrm{d}x}{x^2}\otimes \frac{\mathrm{d}y}{x}\otimes \frac{\mathrm{d}x}{x^2},~\frac{\mathrm{d}y}{x}\otimes \frac{\mathrm{d}x}{x^2}\otimes \frac{\mathrm{d}x}{x^2},\\
&\frac{\mathrm{d}x}{x^2}\otimes \frac{\mathrm{d}y}{x}\otimes \frac{\mathrm{d}y}{x},~~\frac{\mathrm{d}y}{x}\otimes \frac{\mathrm{d}x}{x^2}\otimes \frac{\mathrm{d}y}{x},~\frac{\mathrm{d}y}{x}\otimes \frac{\mathrm{d}y}{x}\otimes \frac{\mathrm{d}x}{x^2},\\
&\frac{\mathrm{d}y}{x}\otimes \frac{\mathrm{d}y}{x}\otimes \frac{\mathrm{d}y}{x}
\end{split}
\]
for 3-tensors, and the basis
\[
\begin{split}
&\frac{\mathrm{d}x}{x^2}\otimes \frac{\mathrm{d}x}{x^2}\otimes \frac{\mathrm{d}x}{x^2}\otimes \frac{\mathrm{d}x}{x^2},\\
&\frac{\mathrm{d}x}{x^2}\otimes \frac{\mathrm{d}x}{x^2}\otimes \frac{\mathrm{d}x}{x^2}\otimes \frac{\mathrm{d}y}{x},~~\frac{\mathrm{d}x}{x^2}\otimes \frac{\mathrm{d}x}{x^2}\otimes \frac{\mathrm{d}y}{x}\otimes \frac{\mathrm{d}x}{x^2},~~\frac{\mathrm{d}x}{x^2}\otimes \frac{\mathrm{d}y}{x}\otimes \frac{\mathrm{d}x}{x^2}\otimes \frac{\mathrm{d}x}{x^2},~~\frac{\mathrm{d}y}{x}\otimes \frac{\mathrm{d}x}{x^2}\otimes \frac{\mathrm{d}x}{x^2}\otimes \frac{\mathrm{d}x}{x^2},\\
&\frac{\mathrm{d}x}{x^2}\otimes \frac{\mathrm{d}x}{x^2}\otimes \frac{\mathrm{d}y}{x}\otimes \frac{\mathrm{d}y}{x},~~\frac{\mathrm{d}x}{x^2}\otimes \frac{\mathrm{d}y}{x}\otimes\frac{\mathrm{d}x}{x^2}\otimes \frac{\mathrm{d}y}{x},~~\frac{\mathrm{d}x}{x^2}\otimes \frac{\mathrm{d}y}{x}\otimes \frac{\mathrm{d}y}{x}\otimes \frac{\mathrm{d}x}{x^2},\\
&~~~~~~~~~~~~~~~~~~~~~~~~~\frac{\mathrm{d}y}{x}\otimes \frac{\mathrm{d}x}{x^2}\otimes \frac{\mathrm{d}x}{x^2}\otimes \frac{\mathrm{d}y}{x},~~\frac{\mathrm{d}y}{x}\otimes \frac{\mathrm{d}x}{x^2}\otimes \frac{\mathrm{d}y}{x}\otimes \frac{\mathrm{d}x}{x^2},~~\frac{\mathrm{d}y}{x}\otimes \frac{\mathrm{d}y}{x}\otimes \frac{\mathrm{d}x}{x^2}\otimes \frac{\mathrm{d}x}{x^2},\\
&\frac{\mathrm{d}y}{x}\otimes \frac{\mathrm{d}y}{x}\otimes \frac{\mathrm{d}x}{x^2}\otimes \frac{\mathrm{d}y}{x},~~\frac{\mathrm{d}y}{x}\otimes \frac{\mathrm{d}y}{x}\otimes \frac{\mathrm{d}y}{x}\otimes \frac{\mathrm{d}x}{x^2},~~\frac{\mathrm{d}y}{x}\otimes \frac{\mathrm{d}x}{x^2}\otimes \frac{\mathrm{d}y}{x}\otimes \frac{\mathrm{d}y}{x},~~\frac{\mathrm{d}x}{x^2}\otimes \frac{\mathrm{d}y}{x}\otimes \frac{\mathrm{d}y}{x}\otimes \frac{\mathrm{d}y}{x},\\
&\frac{\mathrm{d}y}{x}\otimes \frac{\mathrm{d}y}{x}\otimes \frac{\mathrm{d}y}{x}\otimes \frac{\mathrm{d}y}{x},
\end{split}
\]
for 4-tensors. For symmetric 3-tensors, we use the basis
\[\frac{\mathrm{d}x}{x^2}\otimes_s \frac{\mathrm{d}x}{x^2}\otimes_s \frac{\mathrm{d}x}{x^2},~~2\times\frac{\mathrm{d}x}{x^2}\otimes_s \frac{\mathrm{d}x}{x^2}\otimes_s \frac{\mathrm{d}y}{x},~~2\times\frac{\mathrm{d}x}{x^2}\otimes_s \frac{\mathrm{d}y}{x}\otimes_s \frac{\mathrm{d}y}{x},~~\frac{\mathrm{d}y}{x}\otimes_s \frac{\mathrm{d}y}{x}\otimes_s \frac{\mathrm{d}y}{x};\]
for symmetric 4-tensors, we use the basis
\[\begin{split}\frac{\mathrm{d}x}{x^2}\otimes_s\frac{\mathrm{d}x}{x^2}\otimes_s \frac{\mathrm{d}x}{x^2}\otimes_s \frac{\mathrm{d}x}{x^2},~~4\times\frac{\mathrm{d}x}{x^2}\otimes_s\frac{\mathrm{d}x}{x^2}\otimes_s \frac{\mathrm{d}x}{x^2}\otimes_s \frac{\mathrm{d}y}{x},~~6\times\frac{\mathrm{d}x}{x^2}\otimes_s\frac{\mathrm{d}x}{x^2}\otimes_s \frac{\mathrm{d}y}{x}\otimes_s \frac{\mathrm{d}y}{x},\\
~~4\times\frac{\mathrm{d}x}{x^2}\otimes_s\frac{\mathrm{d}y}{x}\otimes_s \frac{\mathrm{d}y}{x}\otimes_s \frac{\mathrm{d}y}{x},~~\frac{\mathrm{d}y}{x}\otimes_s\frac{\mathrm{d}y}{x}\otimes_s \frac{\mathrm{d}y}{x}\otimes_s \frac{\mathrm{d}y}{x}.
\end{split}\]
In the above, $\otimes_s$ denotes the symmetric product, for example, $a\otimes_s b=\mathscr{S}(a\otimes b)$.

\begin{lemma}\label{dd}
On symmetric $4$-tensors, $\mathrm{d}^s_\mathsf{F}\delta^s_\mathsf{F}\in\mathrm{Diff}_{sc}^{2,0}(X;\mathrm{Sym}^{4}{}^{sc}T^*X,\mathrm{Sym}^{4}{}^{sc}T^*X)$ has principal symbol
\[
\begin{split}
\mathfrak{D}(x,y,\xi,\eta)&=\begin{footnotesize}\left(
\begin{array}{cccc}
\xi+\ii \mathsf{F} & 0 & 0 & 0\\
\frac{1}{4}\eta\otimes & \frac{3}{4}(\xi+\ii \mathsf{F}) & 0 & 0\\
a^\flat & \frac{1}{2}\eta\otimes_s & \frac{1}{2}(\xi+\ii \mathsf{F})& 0\\
0 & b^\flat & \frac{3}{4}\eta\otimes_s &\frac{1}{4}(\xi+\ii \mathsf{F})\\
0 & 0& c^\flat & \eta\otimes_s
\end{array}
\right)\left(\begin{array}{ccccc}
\xi-\ii \mathsf{F} & \iota_\eta &6\langle a^\flat,\cdot\rangle &0 &0\\
0 & (\xi-\ii\mathsf{F}) &\iota_\eta^s &\frac{4}{3}\langle b^\flat,\cdot\rangle &0\\
0 & 0 &(\xi-\ii\mathsf{F}) &\iota_\eta^s&\frac{1}{3}\langle c^\flat,\cdot\rangle\\
0 & 0 & 0 &(\xi-\ii\mathsf{F})&\iota_\eta^s
\end{array}
\right)\end{footnotesize}\\
&=\begin{footnotesize}\left(\begin{array}{ccccc}
|\xi|^2+ \mathsf{F}^2 & (\xi+\ii \mathsf{F})\iota_\eta &6(\xi+\ii \mathsf{F})\langle a^\flat,\cdot\rangle & 0 & 0\\
\frac{1}{4}(\xi-\ii\mathsf{F})\eta\otimes & \frac{1}{4}(\eta\otimes)\iota_\eta+\frac{1}{4}(|\xi|^2+\mathsf{F}^2)& \mathfrak{D}_{23} &\mathfrak{D}_{24} & 0\\
(\xi-\ii\mathsf{F})a^\flat & a^\flat\iota_\eta+\frac{1}{2}(\xi-\ii\mathsf{F})\eta\otimes_s &\mathfrak{D}_{33} & \mathfrak{D}_{34}&\mathfrak{D}_{35}\\
0 & (\xi-\ii\mathsf{F})b^\flat&\mathfrak{D}_{43} & \mathfrak{D}_{44}&\mathfrak{D}_{45}\\
0& 0&\mathfrak{D}_{53} & \mathfrak{D}_{54}&\mathfrak{D}_{55}
\end{array}
\right)\end{footnotesize}
\end{split}\]
with
\begin{equation*}
\begin{split}
&\mathfrak{D}_{23}=\frac{3}{2}\eta\otimes\langle a^\flat,\cdot\rangle+\frac{3}{4}(\xi+\ii\mathsf{F})\iota_\eta^s,\\
&\mathfrak{D}_{24}=(\xi+\ii\mathsf{F})\otimes\langle b^\flat,\cdot\rangle,\\
&\mathfrak{D}_{33}=6a^\flat\langle a^\flat,\cdot\rangle+\frac{1}{2}(\eta\otimes)\iota_\eta+\frac{1}{2}(|\xi|^2+\mathsf{F}^2),\\
&\mathfrak{D}_{34}=\frac{2}{3}\eta\otimes\langle b^\flat,\cdot\rangle+\frac{1}{2}(\xi+\ii\mathsf{F})\otimes\iota_\eta^s,\\
&\mathfrak{D}_{35}=\frac{1}{6}(\xi+\ii\mathsf{F})\otimes\langle c^\flat,\cdot\rangle,\\
&\mathfrak{D}_{43}=b^\flat\iota_\eta^s+\frac{3}{4}(\xi-\ii\mathsf{F})\eta\otimes_s,\\
&\mathfrak{D}_{44}=\frac{4}{3}b^\flat\langle b^\flat,\cdot\rangle+\frac{3}{4}(\eta\otimes)\iota_\eta^s+\frac{1}{4}(|\xi|^2+\mathsf{F^2}),\\
&\mathfrak{D}_{45}=\frac{1}{4}\eta\otimes\langle c^\flat,\cdot\rangle+\frac{1}{4}(\xi+\ii\mathsf{F})\otimes\iota_\eta^s,\\
&\mathfrak{D}_{53}=(\xi-\ii\mathsf{F})c^\flat,\\
&\mathfrak{D}_{54}=c^\flat\iota_\eta^s+(\xi-\ii\mathsf{F})\eta\otimes_s,\\
&\mathfrak{D}_{55}=\frac{1}{3}c^\flat\langle c^\flat,\cdot\rangle+\eta\otimes_s\iota_{\eta}.
\end{split}
\end{equation*}
\end{lemma}
The quantities $a^\flat,b^\flat,c^\flat$ are defined in the proof
\begin{proof}
we denote
\[
\begin{split}
f=&f_{xxx}\frac{\mathrm{d}x}{x^2}\otimes_s \frac{\mathrm{d}x}{x^2}\otimes_s \frac{\mathrm{d}x}{x^2}+3\times f_{xxy^i}\frac{\mathrm{d}x}{x^2}\otimes_s \frac{\mathrm{d}x}{x^2}\otimes_s \frac{\mathrm{d}y^i}{x}\\
&+3\times f_{xy^iy^j}\frac{\mathrm{d}x}{x^2}\otimes_s \frac{\mathrm{d}y^i}{x}\otimes_s \frac{\mathrm{d}y^j}{x}+f_{y^iy^jy^k}\frac{\mathrm{d}y^i}{x}\otimes_s \frac{\mathrm{d}y^j}{x}\otimes_s \frac{\mathrm{d}y^k}{x}.
\end{split}
\]
By calculation
\begin{equation*}
\begin{split}
(\nabla f)_{xxxx}&=x^{-6}\partial_xf_{xxx}+O(x^{-7}),\\
(\nabla f)_{xxxy^i}&=x^{-6}\partial_{y^i}f_{xxx}+O(x^{-6}),\\
(\nabla f)_{xxy^ix}&=x^{-5}\partial_xf_{xxy^i}+O(x^{-6}),\\
(\nabla f)_{xxy^iy^j}&=x^{-5}\partial_{y^j}f_{xxy^i}+x^{-6}a_1(f_{xxx})+O(x^{-5}),\\
(\nabla f)_{xy^iy^jx}&=x^{-4}\partial_xf_{xy^iy^j}+O(x^{-5}),\\
(\nabla f)_{xy^iy^jy^k}&=x^{-4}\partial_{y^k}f_{xy^iy^j}+x^{-5}b_1(f_{xxy})+O(x^{-4}),\\
(\nabla f)_{y^iy^jy^kx}&=x^{-3}\partial_xf_{y^iy^jy^k}+O(x^{-4}),\\
(\nabla f)_{y^iy^jy^ky^l}&=x^{-3}\partial_{y^l}f_{y^iy^jy^k}+x^{-4}c_1(f_{xyy})+O(x^{-3}).
\end{split}
\end{equation*}
Here $a_1,b_1,c_1$ come from the contributions of Christoffel symbol $\Gamma$ in equation (\ref{connection}). Then, we derive
\begin{equation}
\begin{split}
\mathrm{d}^sf=&x^2\partial_xf_{xxx}\frac{\mathrm{d}x}{x^2}\otimes_s\frac{\mathrm{d}x}{x^2}\otimes_s \frac{\mathrm{d}x}{x^2}\otimes_s \frac{\mathrm{d}x}{x^2}\\
&+4\times\left(\frac{1}{4}x\partial_{y^i}f_{xxx}+\frac{3}{4}x^2\partial_xf_{xxy^i}\right)\frac{\mathrm{d}x}{x^2}\otimes_s\frac{\mathrm{d}x}{x^2}\otimes_s \frac{\mathrm{d}x}{x^2}\otimes_s \frac{\mathrm{d}y^i}{x}\\
&+6\times\left(\frac{1}{2}\mathrm{Sym}_y(x\partial_{y^j}f_{xxy^i})+\frac{1}{2}x^2\partial_xf_{xy^iy^j}+a^\flat(f_{xxx})\right)\frac{\mathrm{d}x}{x^2}\otimes_s\frac{\mathrm{d}x}{x^2}\otimes_s \frac{\mathrm{d}y^i}{x}\otimes_s \frac{\mathrm{d}y^j}{x}\\
&+4\times\left(\frac{3}{4}\mathrm{Sym}_y(x\partial_{y^k}f_{xy^iy^j})+\frac{1}{4}x^2\partial_xf_{y^iy^jy^k}+b^\flat(f_{xxy})\right)\frac{\mathrm{d}x}{x^2}\otimes_s\frac{\mathrm{d}y^i}{x}\otimes_s \frac{\mathrm{d}y^j}{x}\otimes_s \frac{\mathrm{d}y^k}{x}\\
&+\left({\mathrm{Sym}_y(x\partial_{y^l}f_{y^iy^jy^k}})+c^\flat(f_{xyy})\right)\frac{\mathrm{d}y^i}{x}\otimes_s \frac{\mathrm{d}y^j}{x}\otimes_s \frac{\mathrm{d}y^k}{x}\otimes_s \frac{\mathrm{d}y^l}{x}+\mathrm{l.o.t.}.
\end{split}
\end{equation}
In the above, $\mathrm{Sym}_y$ is defined as
\[\mathrm{Sym}_y(v_{y^{k_1},\cdots,y^{k_m}})=\frac{1}{m!}\sum_{\sigma}v_{y^{k_{\sigma(1)}},\cdots,y^{k_{\sigma(m)}}}.\]
It follows that $\mathrm{d}^s$ has principal symbol
\[
\left(
\begin{array}{cccc}
\xi & 0 & 0 & 0\\
\frac{1}{4}\eta\otimes & \frac{3}{4}\xi & 0 & 0\\
a^\flat & \frac{1}{2}\eta\otimes_s & \frac{1}{2}\xi & 0\\
0 & b^\flat & \frac{3}{4}\eta\otimes_s &\frac{1}{4}\xi\\
0 & 0& c^\flat & \eta\otimes_s
\end{array}
\right).
\]

The term $\eta\otimes_s$ in the $(32)$-block has $(iji')$-entry (corresponding to the $(ij)$ entry of the symmetric 2-tensor on $Y$ and the $i'$ entry of the 1-tensor)
\[\frac{1}{2}(\eta_i\delta_{ji'}+\eta_j\delta_{ii'}).\]
The term $\eta\otimes_s$ in the $(43)$-block has $(ijki'j')$-entry (corresponding to the $(ijk)$ entry of the symmetric 3-tensor and the $i'j'$ entry of the 2-tensor)
\[\frac{1}{6}(\eta_i\delta_{ji'}\delta_{kj'}+\eta_i\delta_{ki'}\delta_{jj'}+\eta_j\delta_{ii'}\delta_{kj'}+\eta_j\delta_{ki'}\delta_{jj'}+\eta_k\delta_{ii'}\delta_{jj'}+\eta_k\delta_{ji'}\delta_{ij'}).\]
The term $\eta\otimes_s$ in the $(54)$-block has $(ijkli'j'k')$-entry (corresponding to the $(ijkl)$ entry of the symmetric 4-tensor and the $i'j'k'$ entry of the 3-tensor)
\[
\begin{split}
\frac{1}{24}(\sum_{\sigma}\eta_i\delta_{j\tau(\sigma(1))}&\delta_{k\tau(\sigma(2))}\delta_{l\tau(\sigma(3))}+\sum_{\sigma}\eta_j\delta_{i\tau(\sigma(1))}\delta_{k\tau(\sigma(2))}\delta_{l\tau(\sigma(3))}\\
&+\sum_{\sigma}\eta_k\delta_{i\tau(\sigma(1))}\delta_{j\tau(\sigma(2))}\delta_{l\tau(\sigma(3))}+\sum_{\sigma}\eta_l\delta_{i\tau(\sigma(1))}\delta_{j\tau(\sigma(2))}\delta_{k\tau(\sigma(3))}).
\end{split}
\]
Here, $\sigma$ runs over all permutations of $(123)$, and $\tau(1)=i',\tau(2)=j',\tau(3)=k'$.

We note that $a^\flat$ maps a 0-tensor (smooth function) to a symmetric 2-tensor, $b^\flat$ maps a symmetric 1-tensor to a symmetric 3-tensor, $c^\flat$ maps a symmetric 2-tensor to a symmetric 4-tensor. They are symmetrizations of $a,b,c$ respectively. Then the symbol of $\mathrm{d}^s_\mathsf{F}=e^{-\mathsf{F}/x}\mathrm{d}^se^{\mathsf{F}/x}$ is given by
\[
\left(
\begin{array}{cccc}
\xi+\ii \mathsf{F} & 0 & 0 & 0\\
\frac{1}{4}\eta\otimes & \frac{3}{4}(\xi+\ii \mathsf{F}) & 0 & 0\\
a^\flat & \frac{1}{2}\eta\otimes_s & \frac{1}{2}(\xi+\ii \mathsf{F}) & 0\\
0 & b^\flat & \frac{3}{4}\eta\otimes_s &\frac{1}{4}(\xi+\ii \mathsf{F})\\
0 &0 & c^\flat & \eta\otimes_s
\end{array}
\right).
\]
We use the inner product
\begin{equation}\label{inner}
M(4)=\left(
\begin{array}{ccccc}
1 & & & &\\
& 4\times\mathrm{Id} &&&\\
&& 6\times\mathrm{Id} &&\\
&&& 4\times\mathrm{Id} &\\
&&&& \mathrm{Id}
\end{array}
\right)
\end{equation}
on symmetric 4-tensors, and 
\begin{equation}\label{inner1}
M(3)=\left(
\begin{array}{ccccc}
1 & & & \\
& 3\times\mathrm{Id} &&\\
&& 3\times\mathrm{Id} &\\
&&& \mathrm{Id} 
\end{array}
\right)
\end{equation}
on symmetric 3-tensors. If $A$ maps a symmetric $m_1$-tensor to a symmetric $m_2$-tensor, we call $B$ the $(m_2,m_1)$-adjoint of $A$ if
\[\langle By,x\rangle_{M(m_1)}=\langle y, Ax\rangle_{M(m_2)}.\]
It is easy to check that 
\[B=M(m_1)^{-1}A^*M(m_2).\]
If $m_1=m_2=m$, we call $A$ is (m,m)-self-adjoint if $B=A$.

It follows that $\delta_\mathsf{F}^s$ has a symbol given by the (3,4)-adjoint of that of $\mathrm{d}^s_\mathsf{F}$,
\[
\left(
\begin{array}{ccccc}
\xi-\ii \mathsf{F} & \iota_\eta &6\langle a^\flat,\cdot\rangle &0 &0\\
0 & (\xi-\ii\mathsf{F}) &\iota_\eta^s &\frac{4}{3}\langle b^\flat,\cdot\rangle &0\\
0 & 0 &(\xi-\ii\mathsf{F}) & \iota_\eta^s&\frac{1}{3}\langle c^\flat,\cdot\rangle\\
0 & 0 & 0 &(\xi-\ii\mathsf{F})&\iota_\eta^s
\end{array}
\right).
\]
Remaining tedious calculations complete the proof.
\end{proof}

\begin{lemma}\label{ellipticity1}
On symmetric $4$-tensors, $N_\mathsf{F}$ is elliptic at fiber infinity in $^{sc}T^*X$ when restricted to the kernel of the principal symbol of $\delta_\mathsf{F}^s$.
\end{lemma}
\begin{proof}
With the notation,
\[S=\frac{X-\alpha(\hat{Y})|Y|^2}{|Y|},~~\hat{Y}=\frac{Y}{|Y|},\]
by (\ref{kernel}), the Schwartz kernel of $N_\mathsf{F}$ at the scattering front face $x=0$ is given by

\begin{equation}\label{NFkernel}
\begin{split}
e^{-\mathsf{F}X}|Y|^{-n+1}\chi(S)\left[S\frac{\mathrm{d}x}{x^2}+\hat{Y}\cdot\frac{\mathrm{d}y}{x}\right]^4\left[(S+2\alpha|Y|)(x^2\partial_x)+\hat{Y}\cdot(x\partial_y)\right]^4.
\end{split}
\end{equation}
On a symmetric 4-tensor of the form
\begin{equation}\label{form}
\begin{split}
f=&f_{xxxx}\frac{\mathrm{d}x}{x^2}\otimes_s\frac{\mathrm{d}x}{x^2}\otimes_s \frac{\mathrm{d}x}{x^2}\otimes_s \frac{\mathrm{d}x}{x^2}+4f_{xxxy}\cdot\frac{\mathrm{d}x}{x^2}\otimes_s\frac{\mathrm{d}x}{x^2}\otimes_s \frac{\mathrm{d}x}{x^2}\otimes_s \frac{\mathrm{d}y}{x}\\
&+6f_{xxyy}\cdot\frac{\mathrm{d}x}{x^2}\otimes_s\frac{\mathrm{d}x}{x^2}\otimes_s \frac{\mathrm{d}y}{x}\otimes_s \frac{\mathrm{d}y}{x}
+4f_{xyyy}\cdot\frac{\mathrm{d}x}{x^2}\otimes_s\frac{\mathrm{d}y}{x}\otimes_s \frac{\mathrm{d}y}{x}\otimes_s \frac{\mathrm{d}y}{x}\\
&+f_{yyyy}\cdot\frac{\mathrm{d}y}{x}\otimes_s\frac{\mathrm{d}y}{x}\otimes_s \frac{\mathrm{d}y}{x}\otimes_s \frac{\mathrm{d}y}{x},
\end{split}
\end{equation}
we have
\[
\begin{split}
&\left[(S+2\alpha|Y|)(x^2\partial_x)+\hat{Y}\cdot(x\partial_y)\right]^4f\\
=&(S+2\alpha |Y|)^4f_{xxxx}+4(S+2\alpha|Y|)^3\langle \hat{Y},f_{xxxy}\rangle+6(S+2\alpha|Y|)^2\langle \hat{Y}\otimes\hat{Y},f_{xxyy}\rangle\\
&+4(S+2\alpha|Y|)\langle \hat{Y}\otimes\hat{Y}\otimes \hat{Y},f_{xyyy}\rangle+\langle \hat{Y}\otimes\hat{Y}\otimes\hat{Y}\otimes\hat{Y},f_{yyyy}\rangle .
\end{split}
\]
On a scalar $a$,
\[
\begin{split}
\left[S\frac{\mathrm{d}x}{x^2}+\hat{Y}\cdot\frac{\mathrm{d}y}{x}\right]^4 a
=a S^4\frac{\mathrm{d}x}{x^2}\otimes_s\frac{\mathrm{d}x}{x^2}\otimes_s \frac{\mathrm{d}x}{x^2}\otimes_s \frac{\mathrm{d}x}{x^2}+4a S^3\hat{Y}\cdot\frac{\mathrm{d}x}{x^2}\otimes_s\frac{\mathrm{d}x}{x^2}\otimes_s \frac{\mathrm{d}x}{x^2}\otimes_s \frac{\mathrm{d}y}{x}\\
+6a S^2\hat{Y}\otimes\hat{Y}\cdot\frac{\mathrm{d}x}{x^2}\otimes_s\frac{\mathrm{d}x}{x^2}\otimes_s \frac{\mathrm{d}y}{x}\otimes_s \frac{\mathrm{d}y}{x}
+4a S\hat{Y}\otimes\hat{Y}\otimes\hat{Y}\cdot\frac{\mathrm{d}x}{x^2}\otimes_s\frac{\mathrm{d}y}{x}\otimes_s \frac{\mathrm{d}y}{x}\otimes_s \frac{\mathrm{d}y}{x}\\
+a\hat{Y}\otimes\hat{Y}\otimes\hat{Y}\otimes\hat{Y}\cdot\frac{\mathrm{d}y}{x}\otimes_s\frac{\mathrm{d}y}{x}\otimes_s \frac{\mathrm{d}y}{x}\otimes_s \frac{\mathrm{d}y}{x}.
\end{split}
\]
Thus, under the basis of symmetric 4-tensors, we have
\[
\begin{split}
&\left[S\frac{\mathrm{d}x}{x^2}+\hat{Y}\cdot\frac{\mathrm{d}y}{x}\right]^4\left[(S+2\alpha|Y|)(x^2\partial_x)+\hat{Y}\cdot(x\partial_y)\right]^4\\
=&\left(
\begin{array}{c}
S^4\\
S^3\hat{Y}\\
S^2\hat{Y}\otimes\hat{Y}\\
S\hat{Y}\otimes\hat{Y}\otimes\hat{Y}\\
\hat{Y}\otimes\hat{Y}\otimes\hat{Y}\otimes\hat{Y}
\end{array}
\right)\otimes
\left(
\begin{array}{c}
(S+2\alpha |Y|)^4 \\
 4 (S+2\alpha |Y|)^3\langle\hat{Y},\cdot\rangle \\
  6 (S+2\alpha |Y|)^2\langle\hat{Y}\otimes\hat{Y},\cdot\rangle \\
 4(S+2\alpha |Y|)\langle\hat{Y}\otimes\hat{Y}\otimes\hat{Y},\cdot\rangle\\
 \langle\hat{Y}\otimes\hat{Y}\otimes\hat{Y}\otimes\hat{Y},\cdot\rangle
\end{array}
\right)^T.
\end{split}
\]
The above matrix is $(4,4)$-self-adjoint.
In coordinates on the support of $\chi$,
\[x,y,|Y|,\frac{X}{|Y|},\hat{Y},\]
we can rewrite the kernel as
\[e^{-\mathsf{F}X}|Y|^{-n+1}\chi(S)\left(
\begin{array}{c}
S^4\\
S^3\hat{Y}\\
S^2\hat{Y}\otimes\hat{Y}\\
S\hat{Y}\otimes\hat{Y}\otimes\hat{Y}\\
\hat{Y}\otimes\hat{Y}\otimes\hat{Y}\otimes\hat{Y}
\end{array}
\right)\otimes
\left(
\begin{array}{c}
(S+2\alpha |Y|)^4 \\
 4 (S+2\alpha |Y|)^3\langle\hat{Y},\cdot\rangle \\
  6 (S+2\alpha |Y|)^2\langle\hat{Y}\otimes\hat{Y},\cdot\rangle \\
 4(S+2\alpha |Y|)\langle\hat{Y}\otimes\hat{Y}\otimes\hat{Y},\cdot\rangle\\
 \langle\hat{Y}\otimes\hat{Y}\otimes\hat{Y}\otimes\hat{Y},\cdot\rangle
\end{array}
\right)^T.\]
The principal symbol associated with $K^\flat$ defined in $(\ref{kernel})$ is the $(X,Y)$-Fourier transform of
\[
\begin{split}
\chi(\tilde{S})|Y|^{-n+1}\left(
\begin{array}{c}
\tilde{S}^4\\
\tilde{S}^3\hat{Y}\\
\tilde{S}^2\hat{Y}\otimes\hat{Y}\\
\tilde{S}\hat{Y}\otimes\hat{Y}\otimes\hat{Y}\\
\hat{Y}\otimes\hat{Y}\otimes\hat{Y}\otimes\hat{Y}
\end{array}
\right)\otimes
\left(
\begin{array}{c}
\tilde{S}^4 \\
 4 \tilde{S}^3\langle\hat{Y},\cdot\rangle \\
  6 \tilde{S}^2\langle\hat{Y}\otimes\hat{Y},\cdot\rangle \\
 4\tilde{S}\langle\hat{Y}\otimes\hat{Y}\otimes\hat{Y},\cdot\rangle\\
 \langle\hat{Y}\otimes\hat{Y}\otimes\hat{Y}\otimes\hat{Y},\cdot\rangle
\end{array}
\right)^T,
\end{split}
\]
with $\tilde{S}=\frac{X}{|Y|}$.
The equatorial sphere is
\begin{equation}\label{equator}
\tilde{S}\xi+\hat{Y}\cdot\eta=0.
\end{equation}
Following the discussion around (3.8) in \cite{UV}, we need to integrate
\begin{equation}\label{proj}
\chi(\tilde{S})\left(
\begin{array}{c}
\tilde{S}^4\\
\tilde{S}^3\hat{Y}\\
\tilde{S}^2\hat{Y}\otimes\hat{Y}\\
\tilde{S}\hat{Y}\otimes\hat{Y}\otimes\hat{Y}\\
\hat{Y}\otimes\hat{Y}\otimes\hat{Y}\otimes\hat{Y}
\end{array}
\right)\otimes
\left(
\begin{array}{c}
\tilde{S}^4 \\
 4 \tilde{S}^3\langle\hat{Y},\cdot\rangle \\
  6 \tilde{S}^2\langle\hat{Y}\otimes\hat{Y},\cdot\rangle \\
 4\tilde{S}\langle\hat{Y}\otimes\hat{Y}\otimes\hat{Y},\cdot\rangle\\
 \langle\hat{Y}\otimes\hat{Y}\otimes\hat{Y}\otimes\hat{Y},\cdot\rangle
\end{array}
\right)^T
\end{equation}
on this sphere. 

For a symmetric 4-tensor of the form (\ref{form}) in the kernel of the principal symbol of $\delta_\mathsf{F}^s$, we have by Lemma \ref{dd} that
\begin{equation}\label{identities1}
\begin{split}
\xi f_{xxxx}+\langle\eta,f_{xxxy}\rangle=0,\\
\xi f_{xxxy}+\langle\eta,f_{xxyy}\rangle=0,\\
\xi f_{xxyy}+\langle\eta,f_{xyyy}\rangle=0,\\
\xi f_{xyyy}+\langle\eta,f_{yyyy}\rangle=0.
\end{split}
\end{equation}
Moreover, $f$ is in the kernel of (\ref{proj}) if and only if
\begin{equation}\label{identities2}
\begin{split}
\tilde{S}^4f_{xxxx}+4\tilde{S}^3\langle\hat{Y},f_{xxxy}\rangle+6\tilde{S}^2\langle\hat{Y}\otimes\hat{Y},f_{xxyy}\rangle\\
+4\tilde{S}\langle\hat{Y}\otimes\hat{Y}\otimes\hat{Y},f_{xyyy}\rangle+\langle\hat{Y}\otimes\hat{Y}\otimes\hat{Y}\otimes\hat{Y},f_{yyyy}\rangle=0.
\end{split}
\end{equation}

Suppose a symmetric 4-tensor $f$ satisfies (\ref{identities1}) and  (\ref{identities2}) for $(\tilde{S},\hat{Y})$ such that $(\ref{equator})$ holds. We will consider two cases, $\xi=0$ and $\xi\neq 0$. \\
\textit{Case 1}: $\xi\neq 0$. If $\eta=0$, we have directly form (\ref{identities1}) that
\[f_{xxxx}~,f_{xxxy},~f_{xxyy},~f_{xyyy}\]
all vanish. Then from $(\ref{identities2})$, we have 
\[
\langle\hat{Y}\otimes\hat{Y}\otimes\hat{Y}\otimes\hat{Y},f_{yyyy}\rangle=0.
\]
Therefore, $f_{yyyy}=0$, since $\hat{Y}\otimes\hat{Y}\otimes\hat{Y}\otimes\hat{Y}$ spans the space of all symmetric 4-tensors with $\eta=0$. If $\eta\neq 0$, we calculate successively,
\begin{equation*}
\begin{split}
&f_{xyyy}=-\langle\frac{\eta}{\xi},f_{yyyy}\rangle,\\
&\langle\hat{Y}\otimes\hat{Y}\otimes\hat{Y},f_{xyyy}\rangle=-\langle\frac{\eta}{\xi}\otimes\hat{Y}\otimes\hat{Y}\otimes\hat{Y},f_{yyyy}\rangle,\\
&f_{xxyy}=-\langle\frac{\eta}{\xi},f_{xyyy}\rangle=\langle\frac{\eta}{\xi}\otimes\frac{\eta}{\xi},f_{yyyy}\rangle,\\
&\langle\hat{Y}\otimes\hat{Y},f_{xxyy}\rangle=\langle\frac{\eta}{\xi}\otimes\frac{\eta}{\xi}\otimes\hat{Y}\otimes\hat{Y},f_{yyyy}\rangle,\\
&f_{xxxy}=-\langle\frac{\eta}{\xi},f_{xxyy}\rangle=-\langle\frac{\eta}{\xi}\otimes\frac{\eta}{\xi}\otimes\frac{\eta}{\xi},f_{yyyy}\rangle,\\
&\langle\hat{Y},f_{xxxy}\rangle=-\langle\frac{\eta}{\xi}\otimes\frac{\eta}{\xi}\otimes\frac{\eta}{\xi}\otimes\hat{Y},f_{yyyy}\rangle,\\
&f_{xxxx}=-\langle\frac{\eta}{\xi},f_{xxxy}\rangle=\langle\frac{\eta}{\xi}\otimes\frac{\eta}{\xi}\otimes\frac{\eta}{\xi}\otimes\frac{\eta}{\xi},f_{yyyy}\rangle.
\end{split}
\end{equation*}
With $\tilde{S}=-\frac{\hat{Y}\cdot\eta}{\xi}$, $(\ref{identities2})$ gives
\begin{equation}\label{identities3}
\begin{split}
\Big{\langle}
\left(\frac{\hat{Y}\cdot\eta}{\xi}\right)^4\frac{\eta}{\xi}\otimes \frac{\eta}{\xi}\otimes \frac{\eta}{\xi}\otimes \frac{\eta}{\xi}+4\left(\frac{\hat{Y}\cdot\eta}{\xi}\right)^3\frac{\eta}{\xi}\otimes \frac{\eta}{\xi}\otimes \frac{\eta}{\xi}\otimes \hat{Y}+6\left(\frac{\hat{Y}\cdot\eta}{\xi}\right)^2\frac{\eta}{\xi}\otimes \frac{\eta}{\xi}\otimes \hat{Y}\otimes \hat{Y}\\
+4\left(\frac{\hat{Y}\cdot\eta}{\xi}\right)\frac{\eta}{\xi}\otimes  \hat{Y}\otimes \hat{Y}\otimes \hat{Y}+\hat{Y}\otimes  \hat{Y}\otimes \hat{Y}\otimes \hat{Y},f_{yyyy}
\Big{\rangle}=0.
\end{split}
\end{equation}
Now we take $\hat{Y}=\epsilon\hat{\eta}+(1-\epsilon^2)^{1/2}\hat{Y}^\perp$, where $\hat{Y}^\perp$ is a unit vector orthogonal to $\hat{\eta}$, and substituting it into (\ref{identities3}), we find that
\begin{equation}\label{identities4}
\begin{split}
\Big{\langle}
&\epsilon^4\left(\frac{|\eta|^8}{\xi^8}+\frac{4|\eta|^6}{\xi^6}+\frac{6|\eta|^4}{\xi^4}+\frac{4|\eta|^2}{\xi^2}+1\right)\hat{\eta} \otimes\hat{\eta}\otimes \hat{\eta}\otimes \hat{\eta}\\
&+4\epsilon^3(1-\epsilon^2)^{1/2}\left(\frac{|\eta|^6}{\xi^6}+\frac{3|\eta|^4}{\xi^4}+\frac{3|\eta|^2}{\xi^2}+1\right)\hat{\eta}\otimes \hat{\eta}\otimes\hat{\eta}\otimes \hat{Y}^\perp\\
&+6\epsilon^2(1-\epsilon^2)\left(\frac{|\eta|^4}{\xi^4}+\frac{2|\eta|^2}{\xi^2}+1\right)\hat{\eta}\otimes \hat{\eta}\otimes \hat{Y}^\perp\otimes \hat{Y}^\perp\\
&+4\epsilon(1-\epsilon^2)^{3/2}\left(\frac{|\eta|^2}{\xi^2}+1\right)\hat{\eta}\otimes  \hat{Y}^\perp\otimes \hat{Y}^\perp\otimes \hat{Y}^\perp\\
&+(1-\epsilon^2)^2\hat{Y}^\perp\otimes  \hat{Y}^\perp\otimes \hat{Y}^\perp\otimes \hat{Y}^\perp,f_{yyyy}
\Big{\rangle}=0.
\end{split}
\end{equation}
Taking $\epsilon=0$ in (\ref{identities4}), we have
\[\langle\hat{Y}^\perp\otimes  \hat{Y}^\perp\otimes \hat{Y}^\perp\otimes \hat{Y}^\perp,f_{yyyy}\rangle=0.\]
Since $\hat{Y}^\perp\otimes  \hat{Y}^\perp\otimes \hat{Y}^\perp\otimes \hat{Y}^\perp$ spans $\eta^\perp\otimes\eta^\perp\otimes\eta^\perp\otimes\eta^\perp$, we conclude that $f_{yyyy}$ is orthogonal to every element of $\eta^\perp\otimes\eta^\perp\otimes\eta^\perp\otimes\eta^\perp$. Taking 1st, 2nd, 3rd and 4th order derivatives of $(\ref{identities4})$ at $\epsilon=0$, it follows that $f_{yyyy}$ is orthogonal to
\[\begin{split}
\hat{\eta}\otimes\hat{\eta}^\perp\otimes\eta^\perp\otimes\hat{\eta}^\perp,~~~~~~\hat{\eta}\otimes\hat{\eta}\otimes\hat{\eta}^\perp\otimes\hat{\eta}^\perp,\\
\hat{\eta}\otimes\hat{\eta}\otimes\hat{\eta}\otimes\hat{\eta}^\perp,~~~~~~~\hat{\eta}\otimes\eta\otimes\hat{\eta}\otimes\hat{\eta},
\end{split}
\]
respectively. We then finally conclude that $f_{yyyy}$ vanishes, and then the whole tensor $f$ vanishes by (\ref{identities1}).

\textit{Case 2}: $\xi=0$ (and so $\eta\neq 0$). Now $(\ref{equator})$ is equivalent to $\eta\cdot\hat{Y}=0$, and (\ref{identities1}) reduces to
\begin{equation}\label{identities5}
\begin{split}
\langle\hat{\eta},f_{xxxy}\rangle=0,\\
\langle\hat{\eta},f_{xxyy}\rangle=0,\\
\langle\hat{\eta},f_{xyyy}\rangle=0,\\
\langle\hat{\eta},f_{yyyy}\rangle=0.
\end{split}
\end{equation}
We differentiate $(\ref{identities2})$ with respect to $\tilde{S}$ up to four times, evaluated at $\tilde{S}=0$, and find that 
\begin{equation}\label{identities6}
\begin{split}
f_{xxxx}=0,\\
\langle\hat{Y},f_{xxxy}\rangle=0,\\
\langle\hat{Y}\otimes\hat{Y},f_{xxyy}\rangle=0,\\
\langle\hat{Y}\otimes\hat{Y}\otimes\hat{Y},f_{xyyy}\rangle=0,\\
\langle\hat{Y}\otimes\hat{Y}\otimes\hat{Y}\otimes\hat{Y},f_{yyyy}\rangle=0.
\end{split}
\end{equation}
Combining the identities in $(\ref{identities5})$ and $(\ref{identities6})$, we conclude that $f=0$.
\end{proof}

\begin{lemma}\label{ellipticity2}
There exists $\mathsf{F}_0>0$ such that on symmetric $4$-tensors
$N_\mathsf{F}$ is elliptic at a finite set of points in $^{sc}T^*X$ when
restricted to the kernel of the principal symbol of
$\delta_\mathsf{F}^s$ for any $\mathsf{F}>\mathsf{F}_0$.
\end{lemma}
\begin{proof}
Taking $\chi(s)=e^{-s^2/(2\nu(\hat{Y}))}$, so $\hat{\chi}(\cdot)=c\sqrt{\nu}e^{-\nu|\cdot|^2/2}$.
We get the $X$-Fourier transform of the Schwartz kernel at the front face $x=0$:
\[
\begin{split}
&\mathcal{F}_XK^\flat(0,y,|Y|,\frac{\xi}{|Y|},\hat{Y})\\
=&|Y|^{2-n}e^{-\ii\alpha(-\xi-\ii\mathsf{F})|Y|^2}\begin{footnotesize}\left(
\begin{array}{c}
D_\sigma^4\\
-D_\sigma^3\hat{Y}\\
D_\sigma^2\hat{Y}\otimes\hat{Y}\\
-D_\sigma\hat{Y}\otimes\hat{Y}\otimes\hat{Y}\\
\hat{Y}\otimes\hat{Y}\otimes\hat{Y}\otimes\hat{Y}
\end{array}
\right)\otimes
\left(
\begin{array}{c}
(-D_\sigma+2\alpha |Y|)^4 \\
 4 (-D_\sigma+2\alpha |Y|)^3\langle\hat{Y},\cdot\rangle \\
  6 (-D_{\sigma}+2\alpha |Y|)^2\langle\hat{Y}\otimes\hat{Y},\cdot\rangle \\
 4(-D_{\sigma}+2\alpha |Y|)\langle\hat{Y}\otimes\hat{Y}\otimes\hat{Y},\cdot\rangle\\
 \langle\hat{Y}\otimes\hat{Y}\otimes\hat{Y}\otimes\hat{Y},\cdot\rangle
\end{array}
\right)^T\end{footnotesize}\hat{\chi}((-\xi-\ii\mathsf{F})|Y|)\\
=&c\sqrt{\nu}|Y|^{2-n}e^{\ii\alpha(\xi+\ii\mathsf{F})|Y|^2}\begin{footnotesize}\left(
\begin{array}{c}
D_\sigma^4\\
-D_\sigma^3\hat{Y}\\
D_\sigma^2\hat{Y}\otimes\hat{Y}\\
-D_\sigma\hat{Y}\otimes\hat{Y}\otimes\hat{Y}\\
\hat{Y}\otimes\hat{Y}\otimes\hat{Y}\otimes\hat{Y}
\end{array}
\right)\otimes
\left(
\begin{array}{c}
(-D_\sigma+2\alpha |Y|)^4 \\
 4 (-D_\sigma+2\alpha |Y|)^3\langle\hat{Y},\cdot\rangle \\
  6 (-D_{\sigma}+2\alpha |Y|)^2\langle\hat{Y}\otimes\hat{Y},\cdot\rangle \\
 4(-D_{\sigma}+2\alpha |Y|)\langle\hat{Y}\otimes\hat{Y}\otimes\hat{Y},\cdot\rangle\\
 \langle\hat{Y}\otimes\hat{Y}\otimes\hat{Y}\otimes\hat{Y},\cdot\rangle
\end{array}
\right)\end{footnotesize}^Te^{-\nu(\xi+\ii\mathsf{F})^2|Y|^2/2}.
\end{split}
\]
Here $D_\sigma$ denotes the differentiation of the argument of $\hat{\chi}$. Then we compute the $Y$-Fourier transform, which in polar coordinates takes the form,
\[
\begin{split}
\int_{\mathbb{S}^{n-2}}\int_0^\infty& e^{-\ii|Y|\hat{Y}\cdot\eta}|Y|^{2-n}e^{\ii\alpha(\xi+\ii\mathsf{F})|Y|^2}\\
&\begin{footnotesize}\left(
\begin{array}{c}
D_\sigma^4\\
-D_\sigma^3\hat{Y}\\
D_\sigma^2\hat{Y}\otimes\hat{Y}\\
-D_\sigma\hat{Y}\otimes\hat{Y}\otimes\hat{Y}\\
\hat{Y}\otimes\hat{Y}\otimes\hat{Y}\otimes\hat{Y}
\end{array}
\right)\otimes
\left(
\begin{array}{c}
(-D_\sigma+2\alpha |Y|)^4 \\
 4 (-D_\sigma+2\alpha |Y|)^3\langle\hat{Y},\cdot\rangle \\
  6 (-D_{\sigma}+2\alpha |Y|)^2\langle\hat{Y}\otimes\hat{Y},\cdot\rangle \\
 4(-D_{\sigma}+2\alpha |Y|)\langle\hat{Y}\otimes\hat{Y}\otimes\hat{Y},\cdot\rangle\\
 \langle\hat{Y}\otimes\hat{Y}\otimes\hat{Y}\otimes\hat{Y},\cdot\rangle
\end{array}
\right)^T\end{footnotesize}e^{-\nu(\xi+\ii\mathsf{F})^2|Y|^2/2}|Y|^{n-2}\mathrm{d}|Y|\mathrm{d}\hat{Y}.
\end{split}
\]
We denote
\[\phi(\xi,\hat{Y})=\nu(\hat{Y})(\xi+\ii\mathsf{F})^2-2\ii\alpha(\xi+\ii\mathsf{F}).\]
By explicitly evaluating the derivates, the above integral yields
\[
\begin{split}
\int_{\mathbb{S}^{n-2}}\int_0^\infty e^{-\ii|Y|\hat{Y}\cdot\eta}\begin{footnotesize}\left(
\begin{array}{c}
\ii^4\nu^4(\xi+\ii\mathsf{F})^4|Y|^4\\
\ii^3\nu^3(\xi+\ii\mathsf{F})^3|Y|^3\hat{Y}\\
\ii^2\nu^2(\xi+\ii\mathsf{F})^2|Y|^2\hat{Y}\otimes\hat{Y}\\
\ii\nu(\xi+\ii\mathsf{F})|Y|\hat{Y}\otimes\hat{Y}\otimes\hat{Y}\\
\hat{Y}\otimes\hat{Y}\otimes\hat{Y}\otimes\hat{Y}
\end{array}
\right)\otimes
\left(
\begin{array}{c}
(\ii\nu(\xi+\ii\mathsf{F})+2\alpha)^4|Y|^4 \\
 4 (\ii\nu(\xi+\ii\mathsf{F})+2\alpha)^3|Y|^3\langle\hat{Y},\cdot\rangle \\
  6 (\ii\nu(\xi+\ii\mathsf{F})+2\alpha)^2|Y|^2\langle\hat{Y}\otimes\hat{Y},\cdot\rangle \\
 4(\ii\nu(\xi+\ii\mathsf{F})+2\alpha)|Y|\langle\hat{Y}\otimes\hat{Y}\otimes\hat{Y},\cdot\rangle\\
 \langle\hat{Y}\otimes\hat{Y}\otimes\hat{Y}\otimes\hat{Y},\cdot\rangle
\end{array}
\right)^T\end{footnotesize}\\
\times e^{-\phi|Y|^2/2}\mathrm{d}|Y|\mathrm{d}\hat{Y}.
\end{split}
\]
We extend the integral in $|Y|$ to $\mathbb{R}$, replacing it by a variable $t$, and using that the integrand is invariant under the joint change of variables $t\rightarrow -t$ and $\hat{Y}\rightarrow-\hat{Y}$. This gives
\[
\begin{split}
\int_{\mathbb{S}^{n-2}}\int_{-\infty}^\infty e^{-\ii t\hat{Y}\cdot\eta}\begin{footnotesize}\left(
\begin{array}{c}
\ii^4\nu^4(\xi+\ii\mathsf{F})^4t^4\\
\ii^3\nu^3(\xi+\ii\mathsf{F})^3t^3\hat{Y}\\
\ii^2\nu^2(\xi+\ii\mathsf{F})^2t^2\hat{Y}\otimes\hat{Y}\\
\ii\nu(\xi+\ii\mathsf{F})t\hat{Y}\otimes\hat{Y}\otimes\hat{Y}\\
\hat{Y}\otimes\hat{Y}\otimes\hat{Y}\otimes\hat{Y}
\end{array}
\right)\otimes
\left(
\begin{array}{c}
(\ii\nu(\xi+\ii\mathsf{F})+2\alpha )^4t^4 \\
 4 (\ii\nu(\xi+\ii\mathsf{F})+2\alpha )^3t^3\langle\hat{Y},\cdot\rangle \\
  6 (\ii\nu(\xi+\ii\mathsf{F})+2\alpha )^2t^2\langle\hat{Y}\otimes\hat{Y},\cdot\rangle \\
 4(\ii\nu(\xi+\ii\mathsf{F})+2\alpha )t\langle\hat{Y}\otimes\hat{Y}\otimes\hat{Y},\cdot\rangle\\
 \langle\hat{Y}\otimes\hat{Y}\otimes\hat{Y}\otimes\hat{Y},\cdot\rangle
\end{array}
\right)^T\end{footnotesize}\\
\times e^{-\phi t^2/2}\mathrm{d}t\mathrm{d}\hat{Y}.
\end{split}
\]
Now the $t$ integral is a Fourier transform evaluated at $-\hat{Y}\cdot\eta$, under which multiplication by $t$ becomes $D_{\hat{Y}\cdot\eta}$. We also note that the Fourier transform of $e^{-\phi(\xi,\hat{Y})t^2/2}$ is a constant multiple of
\begin{equation}
\phi(\xi,\hat{Y})^{-1/2}e^{-(\hat{Y}\cdot\eta)^2/(2\phi(\xi,\hat{Y}))}.
\end{equation}
Thus we are left with
\[
\begin{split}
\int_{\mathbb{S}^{n-2}}\phi(\xi,\hat{Y})^{-1/2}\begin{footnotesize}\left(
\begin{array}{c}
\ii^4\nu^4(\xi+\ii\mathsf{F})^4D_{\hat{Y}\cdot\eta}^4\\
\ii^3\nu^3(\xi+\ii\mathsf{F})^3D_{\hat{Y}\cdot\eta}^3\hat{Y}\\
\ii^2\nu^2(\xi+\ii\mathsf{F})^2D_{\hat{Y}\cdot\eta}^2\hat{Y}\otimes\hat{Y}\\
\ii\nu(\xi+\ii\mathsf{F})D_{\hat{Y}\cdot\eta}\hat{Y}\otimes\hat{Y}\otimes\hat{Y}\\
\hat{Y}\otimes\hat{Y}\otimes\hat{Y}\otimes\hat{Y}
\end{array}
\right)\otimes
\left(
\begin{array}{c}
(\ii\nu(\xi+\ii\mathsf{F})+2\alpha )^4D_{\hat{Y}\cdot\eta}^4 \\
 4 (\ii\nu(\xi+\ii\mathsf{F})+2\alpha )^3\langle\hat{Y},\cdot\rangle D_{\hat{Y}\cdot\eta}^3 \\
  6 (\ii\nu(\xi+\ii\mathsf{F})+2\alpha )^2\langle\hat{Y}\otimes\hat{Y},\cdot\rangle D_{\hat{Y}\cdot\eta}^2 \\
 4(\ii\nu(\xi+\ii\mathsf{F})+2\alpha )\langle\hat{Y}\otimes\hat{Y}\otimes\hat{Y},\cdot\rangle D_{\hat{Y}\cdot\eta}\\
 \langle\hat{Y}\otimes\hat{Y}\otimes\hat{Y}\otimes\hat{Y},\cdot\rangle
\end{array}
\right)^T\end{footnotesize}\\
\times e^{-(\hat{Y}\cdot\eta)^2/(2\phi(\xi,\hat{Y}))}\mathrm{d}\hat{Y}.\\
=\int_{\mathbb{S}^{n-2}}\phi(\xi,\hat{Y})^{-1/2}\begin{footnotesize}\left(
\begin{array}{c}
\nu^4(\xi+\ii\mathsf{F})^4(\frac{\hat{Y}\cdot\eta}{\phi})^4\\
-\nu^3(\xi+\ii\mathsf{F})^3(\frac{\hat{Y}\cdot\eta}{\phi})^3\hat{Y}\\
\nu^2(\xi+\ii\mathsf{F})^2(\frac{\hat{Y}\cdot\eta}{\phi})^2\hat{Y}\otimes\hat{Y}\\
-\nu(\xi+\ii\mathsf{F})(\frac{\hat{Y}\cdot\eta}{\phi})\hat{Y}\otimes\hat{Y}\otimes\hat{Y}\\
\hat{Y}\otimes\hat{Y}\otimes\hat{Y}\otimes\hat{Y}
\end{array}
\right)\otimes
\left(
\begin{array}{c}
(\nu(\xi+\ii\mathsf{F})-2\ii\alpha )^4(\frac{\hat{Y}\cdot\eta}{\phi})^4 \\
-4 (\nu(\xi+\ii\mathsf{F})-2\ii\alpha )^3 (\frac{\hat{Y}\cdot\eta}{\phi})^3\langle\hat{Y},\cdot\rangle \\
  6 (\nu(\xi+\ii\mathsf{F})-2\ii\alpha )^2(\frac{\hat{Y}\cdot\eta}{\phi})^2 \langle\hat{Y}\otimes\hat{Y},\cdot\rangle \\
 -4(\nu(\xi+\ii\mathsf{F})-2\ii\alpha ) (\frac{\hat{Y}\cdot\eta}{\phi})\langle\hat{Y}\otimes\hat{Y}\otimes\hat{Y},\cdot\rangle\\
 \langle\hat{Y}\otimes\hat{Y}\otimes\hat{Y}\otimes\hat{Y},\cdot\rangle
\end{array}
\right)^T\end{footnotesize}\\
\times e^{-(\hat{Y}\cdot\eta)^2/(2\phi(\xi,\hat{Y}))}\mathrm{d}\hat{Y}.\\
\end{split}
\]

We note that
\begin{equation}\label{proj1}
\begin{footnotesize}\left(
\begin{array}{c}
\nu^4(\xi+\ii\mathsf{F})^4(\frac{\hat{Y}\cdot\eta}{\phi})^4\\
-\nu^3(\xi+\ii\mathsf{F})^3(\frac{\hat{Y}\cdot\eta}{\phi})^3\hat{Y}\\
\nu^2(\xi+\ii\mathsf{F})^2(\frac{\hat{Y}\cdot\eta}{\phi})^2\hat{Y}\otimes\hat{Y}\\
-\nu(\xi+\ii\mathsf{F})(\frac{\hat{Y}\cdot\eta}{\phi})\hat{Y}\otimes\hat{Y}\otimes\hat{Y}\\
\hat{Y}\otimes\hat{Y}\otimes\hat{Y}\otimes\hat{Y}
\end{array}
\right)\otimes
\left(
\begin{array}{c}
(\nu(\xi+\ii\mathsf{F})-2\ii\alpha )^4(\frac{\hat{Y}\cdot\eta}{\phi})^4 \\
-4 (\nu(\xi+\ii\mathsf{F})-2\ii\alpha )^3 (\frac{\hat{Y}\cdot\eta}{\phi})^3\langle\hat{Y},\cdot\rangle \\
  6 (\nu(\xi+\ii\mathsf{F})-2\ii\alpha )^2(\frac{\hat{Y}\cdot\eta}{\phi})^2 \langle\hat{Y}\otimes\hat{Y},\cdot\rangle \\
 -4(\nu(\xi+\ii\mathsf{F})-2\ii\alpha ) (\frac{\hat{Y}\cdot\eta}{\phi})\langle\hat{Y}\otimes\hat{Y}\otimes\hat{Y},\cdot\rangle\\
 \langle\hat{Y}\otimes\hat{Y}\otimes\hat{Y}\otimes\hat{Y},\cdot\rangle
\end{array}
\right)^T\end{footnotesize}
\end{equation}
is a multiple of a projection and is $(4,4)$-self-adjoint.
We let $\nu=\mathsf{F}^{-1}\alpha$, with
\[\nu(\xi+\ii\mathsf{F})-2\ii\alpha=\nu(\xi-\ii\mathsf{F}),\]
and
\[\phi=(\xi+\ii\mathsf{F})(\nu(\xi+\ii\mathsf{F})-2\ii\alpha)=\nu(\xi^2+\mathsf{F}^2).\]
We then denote
\[
\begin{split}
&\mathfrak{C}_4=\nu^4(\xi+\ii\mathsf{F})-2\ii\alpha )^4(\frac{\hat{Y}\cdot\eta}{\phi})^4=\nu^4(\xi-\ii\mathsf{F})^4(\frac{\hat{Y}\cdot\eta}{\phi})^4,\\
&\mathfrak{C}_3=-\nu^3(\xi+\ii\mathsf{F})-2\ii\alpha )^3(\frac{\hat{Y}\cdot\eta}{\phi})^3=-\nu^3(\xi-\ii\mathsf{F})^3(\frac{\hat{Y}\cdot\eta}{\phi})^3,\\
&\mathfrak{C}_2=\nu^2(\xi+\ii\mathsf{F})-2\ii\alpha )^2(\frac{\hat{Y}\cdot\eta}{\phi})^2=\nu^2(\xi-\ii\mathsf{F})^2(\frac{\hat{Y}\cdot\eta}{\phi})^2,\\
&\mathfrak{C}_1=-\nu(\xi+\ii\mathsf{F})-2\ii\alpha)(\frac{\hat{Y}\cdot\eta}{\phi})=-\nu(\xi-\ii\mathsf{F})(\frac{\hat{Y}\cdot\eta}{\phi}).\\
\end{split}
\]

For a symmetric 4-tensor of the form (\ref{form}) in the kernel of the principal symbol of $\delta_\mathsf{F}^s$, we have by Lemma \ref{dd} that
\begin{equation}\label{identities1m}
\begin{split}
&(\xi-\ii\mathsf{F}) f_{xxxx}+\langle\eta,f_{xxxy}\rangle+6\langle a^\flat,f_{xxyy}\rangle=0,\\
&(\xi-\ii\mathsf{F}) f_{xxxy}+\langle\eta,f_{xxyy}\rangle+4\langle b^\flat,f_{xyyy}\rangle=0,\\
&(\xi-\ii\mathsf{F}) f_{xxyy}+\langle\eta,f_{xyyy}\rangle+\langle c^\flat,f_{yyyy}\rangle=0,\\
&(\xi-\ii\mathsf{F}) f_{xyyy}+\langle\eta,f_{yyyy}\rangle=0.
\end{split}
\end{equation}
Moreover, $f$ is in the kernel of (\ref{proj1}) if and only if
\begin{equation}\label{identities2m}
\begin{split}
\mathfrak{C}_4f_{xxxx}&+4\mathfrak{C}_3\langle\hat{Y},f_{xxxy}\rangle+6\mathfrak{C}_2\langle\hat{Y}\otimes\hat{Y},f_{xxyy}\rangle\\
&+4\mathfrak{C}_1\langle\hat{Y}\otimes\hat{Y}\otimes\hat{Y},f_{xyyy}\rangle+\langle\hat{Y}\otimes\hat{Y}\otimes\hat{Y}\otimes\hat{Y},f_{yyyy}\rangle=0.
\end{split}
\end{equation}

We now take a semiclassical point of viewm setting $h=\mathsf{F}^{-1}$ and rescaling
\[\xi_\mathsf{F}=\mathsf{F}^{-1}\xi,~~\eta_\mathsf{F}=\mathsf{F}^{-1}\eta.\]
Using these semiclassical variables, we calculate, successively,
\begin{equation*}
\begin{split}
&f_{xyyy}=-(\xi_\mathsf{F}-\ii)^{-1}\langle\eta_\mathsf{F},f_{yyyy}\rangle,\\
&\langle\hat{Y}\otimes\hat{Y}\otimes\hat{Y},f_{xyyy}\rangle=-(\xi_\mathsf{F}-\ii)^{-1}\langle\eta_\mathsf{F}\otimes\hat{Y}\otimes\hat{Y}\otimes\hat{Y},f_{yyyy}\rangle,\\
&f_{xxyy}=-(\xi_\mathsf{F}-\ii)^{-1}\langle\eta_\mathsf{F},f_{xyyy}\rangle+O(h)=(\xi_\mathsf{F}-\ii)^{-2}\langle\eta_\mathsf{F}\otimes\eta_\mathsf{F},f_{yyyy}\rangle+O(h),\\
&\langle\hat{Y}\otimes\hat{Y},f_{xxyy}\rangle=(\xi_\mathsf{F}-\ii)^{-2}\langle\eta_\mathsf{F}\otimes\eta_\mathsf{F}\otimes\hat{Y}\otimes\hat{Y},f_{yyyy}\rangle+O(h),\\
&f_{xxxy}=-(\xi_\mathsf{F}-\ii)^{-1}\langle\eta_\mathsf{F},f_{xxyy}\rangle+O(h)=-(\xi_\mathsf{F}-\ii)^{-3}\langle\eta_\mathsf{F}\otimes\eta_\mathsf{F}\otimes\eta_\mathsf{F},f_{yyyy}\rangle+O(h),\\
&\langle\hat{Y},f_{xxxy}\rangle=-(\xi_\mathsf{F}-\ii)^{-3}\langle\eta_\mathsf{F}\otimes\eta_\mathsf{F}\otimes\eta_\mathsf{F}\otimes\hat{Y},f_{yyyy}\rangle+O(h),\\
&f_{xxxx}=-(\xi_\mathsf{F}-\ii)^{-1}\langle\eta_\mathsf{F},f_{xxxy}\rangle+O(h)=(\xi_\mathsf{F}-\ii)^{-4}\langle\eta_\mathsf{F}\otimes\eta_\mathsf{F}\otimes\eta_\mathsf{F}\otimes\eta_\mathsf{F},f_{yyyy}\rangle.
\end{split}
\end{equation*}
we observe that
\begin{eqnarray*}
\mathfrak{C}_j=(-1)^{j}(\xi_\mathsf{F}^2+1)^{-j}(\xi_\mathsf{F}-\ii)^{j}\rho^j,~~~\text{with}~\rho=\hat{Y}\cdot\eta_\mathsf{F}.
\end{eqnarray*}
By calculation, and letting $h\rightarrow 0$, we have by (\ref{identities1m}) that
\[
\langle \otimes_{i=1}^4((\xi_\mathsf{F}^2+1)^{-1}\rho\eta_\mathsf{F}+\hat{Y}),f_{yyyy}\rangle=0.
\]
If $\eta=0$, then
\[
\langle \hat{Y}\otimes\hat{Y}\otimes\hat{Y}\otimes\hat{Y},f_{yyyy}\rangle=0.
\]
Since $\hat{Y}\otimes\hat{Y}\otimes\hat{Y}\otimes\hat{Y}$ span the space of all symmetric 4-tensors, we conclude that $f_{yyyy}=0$ and thus $f=0$.

If $\eta_\mathsf{F}\neq 0$, we take $\hat{Y}=\epsilon\hat{\eta}_{\mathsf{F}}+(1-\epsilon^2)^{1/2}\hat{Y}^\perp$, where $\hat{Y}^\perp$ is orthogonal to $\hat{\eta}_{\mathsf{F}}$. Then by (\ref{identities1m}), we have 
\begin{equation}\label{identities4m}
\begin{split}
\Big{\langle}&
\epsilon^4\left(1+\frac{|\eta_\mathsf{F}|^2}{\xi^2_\mathsf{F}+1}\right)^4 \hat{\eta}_\mathsf{F} \otimes\hat{\eta}_\mathsf{F}\otimes \hat{\eta}_\mathsf{F}\otimes \hat{\eta}_\mathsf{F}\\
&+4\epsilon^3(1-\epsilon^2)^{1/2}\left(1+\frac{|\eta_\mathsf{F}|^2}{\xi^2_\mathsf{F}+1}\right)^3\hat{\eta}_\mathsf{F}\otimes \hat{\eta}_\mathsf{F}\otimes\hat{\eta}_\mathsf{F}\otimes \hat{Y}^\perp\\
&+6\epsilon^2(1-\epsilon^2)\left(1+\frac{|\eta_\mathsf{F}|^2}{\xi^2_\mathsf{F}+1}\right)^2\hat{\eta}_\mathsf{F}\otimes \hat{\eta}_\mathsf{F}\otimes \hat{Y}^\perp\otimes \hat{Y}^\perp\\
&+4\epsilon(1-\epsilon^2)^{3/2}\left(1+\frac{|\eta_\mathsf{F}|^2}{\xi^2_\mathsf{F}+1}\right)\hat{\eta}_\mathsf{F}\otimes  \hat{Y}^\perp\otimes \hat{Y}^\perp\otimes \hat{Y}^\perp\\
&+(1-\epsilon^2)^2\hat{Y}^\perp\otimes  \hat{Y}^\perp\otimes \hat{Y}^\perp\otimes \hat{Y}^\perp,f_{yyyy}
\Big{\rangle}=0.
\end{split}
\end{equation}
Similar to the proof of Lemma \ref{ellipticity1}, we take derivatives of $(\ref{identities4m})$ up to order four at $\epsilon=0$; it follows that $f_{yyyy}$ is orthogonal to
\[\begin{split}
\hat{Y}^\perp\otimes\hat{Y}^\perp\otimes\hat{Y}^\perp\otimes\hat{Y}^\perp,~~~~~\hat{\eta}_\mathsf{F}\otimes\hat{Y}^\perp\otimes\hat{Y}^\perp\otimes\hat{Y}^\perp,~~~~~~\hat{\eta}_\mathsf{F}\otimes\hat{\eta}_\mathsf{F}\otimes\hat{Y}^\perp\otimes\hat{Y}^\perp,\\
\hat{\eta}_\mathsf{F}\otimes\hat{\eta}_\mathsf{F}\otimes\hat{\eta}_\mathsf{F}\otimes\hat{Y}^\perp,~~~~~~~\hat{\eta}_\mathsf{F}\otimes\hat{\eta}_\mathsf{F}\otimes\hat{\eta}_\mathsf{F}\otimes\hat{\eta}_\mathsf{F}.
\end{split}
\]
Then, $f=0$.

We conclude that for sufficiently large $\mathsf{F}>0$, one has
ellipticity at all finite points. \end{proof}

With Lemma \ref{ellipticity1} and Lemma \ref{ellipticity2}, we obtain
the following proposition by similar arguments as in the proof of
\cite[Proposition 3.3]{SUV2},

\begin{proposition}
There exists $\mathsf{F}_0>0$ such that for $\mathsf{F}>\mathsf{F}_0$ the following holds. Given $\tilde{\Omega}$, a neighborhood of $X\cap M=\{x\geq 0,\rho>0\}$ in $X$; for a suitable choice of the cutoff $\chi\in C_c^\infty(\mathbb{R})$ and of $M\in\Psi_{sc}^{-3,0}(X;\mathrm{Sym}^{3}{}^{sc}T^*X,\mathrm{Sym}^{3}{}^{sc}T^*X)$, the operator
\[A_\mathsf{F}=N_\mathsf{F}+\mathrm{d}^s_\mathsf{F}M\delta^s_\mathsf{F},~~~N_\mathsf{F}=e^{-\mathsf{F}/x}LI_4e^{\mathsf{F}/x},~~\mathrm{d}^s_\mathsf{F}=e^{-\mathsf{F}/x}\mathrm{d}^se^{\mathsf{F}/x},\]
is elliptic in $\Psi_{sc}^{-1,0}(X;\mathrm{Sym}^{4}{}^{sc}T^*X,\mathrm{Sym}^{4}{}^{sc}T^*X)$ in $\tilde{\Omega}$.
\end{proposition}

\section{Proofs of the main results}\label{mainproof}

We prove the injectivity of $I_4$ with the gauge condition
$\delta_\mathsf{F}^sf_\mathsf{F}=0$ in $\Omega=\Omega_c$, where
$f_\mathsf{F}=e^{-\mathsf{F}/x}f$. Based on the discussion in
\cite[Section 4]{SUV2}, we first need to check the invertibility of
$\Delta_{\mathsf{F},s}$. Here
$\Delta_{\mathsf{F},s}=\delta_\mathsf{F}^s\mathrm{d}_\mathsf{F}^s$ is
the `solenoidal Witten Laplacian' which we will show to be invertible
with the desired boundary condition. The similar results for $I_1$ and
$I_2$ are provided in Section 4 of \cite{SUV2}.

\begin{lemma}\label{dd1}
There exists $\mathsf{F}_0>0$ such that for $\mathsf{F}\geq\mathsf{F}_0$ the operator $\Delta_{F}^s=\delta^s_\mathsf{F}\mathrm{d}^s_\mathsf{F}$ is (joint) elliptic in $\mathrm{Diff}_{sc}^{2,0}(X;\mathrm{Sym}^{3}{}^{sc}T^*X,\mathrm{Sym}^{3}{}^{sc}T^*X)$ on symmetric $3$-tensors. In fact, on symmetric $3$-tensors
\begin{equation}\label{fact1}
\delta^s_\mathsf{F}\mathrm{d}^s_\mathsf{F}=\frac{1}{4}\nabla_\mathsf{F}^*\nabla_\mathsf{F}+\frac{3}{4}\mathrm{d}^s_\mathsf{F}\delta^s_\mathsf{F}+A+R,
\end{equation}
where $R\in x\mathrm{Diff}^1_{sc}(X;\mathrm{Sym}^{3}{}^{sc}T^*X,\mathrm{Sym}^{3}{}^{sc}T^*X)$, $A\in \mathrm{Diff}^1_{sc}(X;\mathrm{Sym}^{3}{}^{sc}T^*X,\mathrm{Sym}^{3}{}^{sc}T^*X)$ and $\nabla_\mathsf{F}=e^{-\mathsf{F}/x}\nabla e^{\mathsf{F}/x}$, with $\nabla$ gradient relative to $g_{sc}$ (not $g$), $\mathrm{d}_\mathsf{F}=e^{-\mathsf{F}/x}\mathrm{d} e^{\mathsf{F}/x}$ the exterior derivative on symmetric $3$-tensors, while $\delta_\mathsf{F}$ is its adjoint on symmetric $3$-tensors.
\end{lemma}
\begin{proof}
By calculation and Lemma \ref{dd}, $\Delta_\mathsf{F}^s$ has symbol
\[
\begin{split}
\left(
\begin{array}{cccc}
\xi^2+\mathsf{F}^2+\frac{1}{4}|\eta|^2& \frac{3}{4}(\xi+\ii\mathsf{F})\iota_\eta&  0&0\\
\frac{1}{4}(\xi-\ii\mathsf{F})\eta\otimes & \frac{3}{4}(\xi^2+\mathsf{F}^2)+\frac{1}{2}\iota_\eta^s\eta\otimes_s& \frac{1}{2}(\xi+\ii\mathsf{F})\iota^s_\eta &0 \\
0 & \frac{1}{2}(\xi-\ii\mathsf{F})\eta\otimes_s &\frac{1}{2}(\xi^2+\mathsf{F}^2)+\frac{3}{4}\iota_\eta^s\eta\otimes_s &\frac{1}{4}(\xi+\ii\mathsf{F})\iota^s_\eta\\
0 & 0&\frac{3}{4}(\xi-\ii\mathsf{F})\eta\otimes_s &\frac{1}{4}(\xi^2+\mathsf{F}^2)+\iota^s_\eta\eta\otimes_s
\end{array}
\right)\\
+\left(
\begin{array}{cccc}
6\langle a^\flat,\cdot\rangle a^\flat&3\langle a,\cdot\rangle\eta\otimes_s & 3(\xi+\ii\mathsf{F})\langle a^\flat,\cdot\rangle &0\\
\iota^s_\eta a & \frac{4}{3}\langle b^\flat,\cdot\rangle b^\flat & \langle b^\flat,\cdot\rangle\eta\otimes_s & \frac{1}{3}(\xi+\ii\mathsf{F})\langle b^\flat,\cdot\rangle\\
(\xi-\ii \mathsf{F})a^\flat & \iota_\eta^s b^\flat &\langle c^\flat,\cdot\rangle c^\flat &\frac{1}{3}\langle c^\flat,\cdot\rangle\eta\otimes_s\\
0 & (\xi-\ii\mathsf{F})b^\flat &\iota^s_\eta c^\flat &0
\end{array}
\right).\\
\end{split}
\]
Here, $\iota_\eta^s\eta\otimes_s$ at $(2,2)$-block has the $(i_1',i_2')$-entry
\[\frac{1}{2}(|\eta|^2\delta_{i_1',i_2'}+\eta_{i_1'}\eta_{i_2'}).\]
$\iota_\eta^s\eta\otimes_s$ at $(3,3)$-block has $(i_1',j_1',i_2',j_2')$-entry
\[
\begin{split}
\frac{1}{6}(|\eta|^2\delta_{i_1',i_2'}\delta_{j_1',j_2'}+|\eta|^2\delta_{i_1',j_2'}\delta_{j_1',i_2'}+\eta_{i_1'}\eta_{i_2'}\delta_{j_1'j_2'}+\eta_{i_1'}\eta_{j_2'}\delta_{i_1'j_2'}+\eta_{i_2'}\eta_{j_1'}\delta_{i_1'j_2'}+\eta_{j_1'}\eta_{j_2'}\delta_{i_1'i_2'})
\end{split}
\]
and $\iota_\eta^s\eta\otimes_s$ at $(4,4)$-block has $(i_1',j_1',k_1',i_2',j_2',k_2')$-entry
\[
\begin{split}
\frac{1}{24}\Big(
&|\eta|^2(\delta_{i_1'i_2'}\delta_{j_1'j_2'}\delta_{k_1'k_2'}+\delta_{i_1'j_2'}\delta_{j_1'i_2'}\delta_{k_1'k_2'}+\delta_{i_1'k_2'}\delta_{j_1'j_2'}\delta_{k_1'i_2'}\\
&\quad\quad\quad+\delta_{i_1'i_2'}\delta_{j_1'k_2'}\delta_{k_1'j_2'}+\delta_{i_1'j_2'}\delta_{j_1'k_2'}\delta_{k_1'i_2'}+\delta_{i_1'k_2'}\delta_{j_1'i_2'}\delta_{k_1'j_2'})\\
&+\eta_{i_1'}\eta_{i_2'}(\delta_{j_1'j_2'}\delta_{k_1'k_2'}+\delta_{j_1'k_2'}\delta_{k_1'j_2'})+\eta_{i_1'}\eta_{j_2'}(\delta_{j_1'i_2'}\delta_{k_1'k_2'}+\delta_{j_1'k_2'}\delta_{i_1'k_2'})\\
&+\eta_{i_1'}\eta_{k_2'}(\delta_{j_1'j_2'}\delta_{k_1'i_2'}+\delta_{j_1'i_2'}\delta_{k_1'j_2'})+\eta_{j_1'}\eta_{k_2'}(\delta_{i_1'j_2'}\delta_{k_1'i_2'}+\delta_{i_1'i_2'}\delta_{k_1'j_2'})\\
&+\eta_{j_1'}\eta_{j_2'}(\delta_{i_1'i_2'}\delta_{k_1'k_2'}+\delta_{i_1'k_2'}\delta_{i_1'k_2'})+\eta_{k_1'}\eta_{k_2'}(\delta_{i_1'i_2'}\delta_{j_1'j_2'}+\delta_{i_1'j_2'}\delta_{j_1'i_2'})\\
&+\eta_{j_1'}\eta_{i_2'}(\delta_{i_1'j_2'}\delta_{k_1'k_2'}+\delta_{i_1'k_2'}\delta_{j_1'k_2'})+\eta_{k_1'}\eta_{i_2'}(\delta_{j_1'j_2'}\delta_{i_1'k_2'}+\delta_{j_1'k_2'}\delta_{i_1'j_2'})\\
&+\eta_{j_1'}\eta_{j_2'}(\delta_{i_1'k_2'}\delta_{j_1'i_2'}+\delta_{i_1'i_2'}\delta_{j_1'k_2'})
\Big) .
\end{split}
\]
We note that the gradient $\nabla$ maps a symmetric 3-tensor to a (not
symmetric) 4-tensor.

We introduce some further notation. We let $A$ be
a matrix of blocks, with
\[A_{\downarrow\times k}\]
representing
\[
\left.\left(
\begin{array}{c}
A\\
\vdots\\
A
\end{array}
\right)\right\}k-\mathrm{tuple}.
\]
Also, we write
\[A_{\rightarrow\times k}\]
representing
\[
\left(
\begin{array}{ccc}
A & \cdots&A
\end{array}
\right).
\]
Then we use the basis for 4-tensors (not the symmetric ones) and symmetric 3-tensors, under which the principal symbol of $\nabla_\mathsf{F}$ relative to $g_{sc}$ (not $g$) is
\[
\left(
\begin{array}{cccc}
\left(
\begin{array}{c}
\xi+\ii\mathsf{F}\\
\eta\otimes
\end{array}
\right)
&&&\\

& \left(
\begin{array}{c}
\xi+\ii\mathsf{F}\\
\eta\otimes
\end{array}
\right)_{\downarrow\times 3} &&\\
 
 &&\left(
\begin{array}{c}
\xi+\ii\mathsf{F}\\
\eta\otimes
\end{array}
\right)_{\downarrow\times 3}&\\
&&&\left(
\begin{array}{c}
\xi+\ii\mathsf{F}\\
\eta\otimes
\end{array}
\right)

 \end{array}
 \right) .
\]
The number of rows is 16. Thus $\nabla_\mathsf{F}^*$ has principal symbol,
\[
\left(
\begin{array}{cccc}
(
\begin{array}{cc}
\xi-\ii\mathsf{F} & \iota_\eta
\end{array}
)&&&\\
&
\frac{1}{3}(
\begin{array}{cc}
\xi-\ii\mathsf{F} & \iota_\eta
\end{array}
)_{\rightarrow\times 3}
&&\\
&&
\frac{1}{3}(
\begin{array}{cc}
\xi-\ii\mathsf{F} & \iota_\eta
\end{array}
)_{\rightarrow\times 3}
&\\
&&&
(
\begin{array}{cc}
\xi-\ii\mathsf{F} & \iota_\eta
\end{array}
)
\end{array}
\right).
\]
Then $\nabla_\mathsf{F}^*\nabla_\mathsf{F}$ has symbol
\begin{equation}\label{FFsymbol}
\left(
\begin{array}{cccc}
\xi^2+\mathsf{F}^2+|\eta|^2& 0&  0&0\\
0&\xi^2+\mathsf{F}^2+|\eta|^2& 0 & 0 \\
0 & 0 &\xi^2+\mathsf{F}^2+|\eta|^2&0\\
0 & 0 & 0 &\xi^2+\mathsf{F}^2+|\eta|^2
\end{array}
\right).
\end{equation}

Similar to our calculation in the proof of Lemma \ref{dd}, we get the principal symbol of $\mathrm{d}^s_\mathsf{F}\delta^s_\mathsf{F}$ on symmetric 3-tensors,
\[
\begin{split}
\left(
\begin{array}{cccc}
\xi^2+\mathsf{F}^2 & (\xi+\ii\mathsf{F})\iota_\eta & 0 & 0\\
\frac{1}{3}(\xi-\ii\mathsf{F})\eta\otimes &\frac{2}{3}(\xi^2+\mathsf{F}^2)+\frac{1}{3}\eta\otimes\iota_\eta & \frac{2}{3}(\xi+\ii\mathsf{F})\iota_\eta^s & 0\\
0 &\frac{2}{3}(\xi-\ii\mathsf{F})\eta\otimes_s & \frac{1}{3}(\xi+\mathsf{F}^2)+\frac{2}{3}\eta\otimes_s\iota_\eta & \frac{1}{3}(\xi+\mathsf{F})\iota_\eta^s\\
0 & 0 & (\xi-\ii\mathsf{F})\eta\otimes_s & \eta\otimes_s\iota_\eta
\end{array}
\right)\\
+
\left(
\begin{array}{cccc}
0 & 0 & 3(\xi+\ii\mathsf{F})\langle d^\flat, \cdot\rangle & 0\\
0 & 0 & \eta\langle d^\flat,\cdot\rangle &\frac{1}{3}(\xi+\ii\mathsf{F})\langle e^\flat,\cdot\rangle\\
(\xi-\ii\mathsf{F})d^\flat & d^\flat\iota_\eta & 3d^\flat\langle d^\flat,\cdot\rangle & \frac{1}{3}\eta\otimes_s\langle e^\flat,\cdot\rangle\\
0 & (\xi+\ii\mathsf{F})e^\flat & e^\flat\iota_\eta^s & \frac{1}{2}e^\flat\langle e^\flat,\cdot\rangle
\end{array}
\right).
\end{split}
\]
Here, $\eta\otimes\iota_\eta$ at the $(2,2)$-block has $(i_1',i_2')$-entry
\[\eta_{i_1',i_2'} ,\]
$\eta\otimes_s\iota_\eta$ at the $(3,3)$-block has $(i_1',j_1',i_2',j_2')$-entry
\[\frac{1}{4}(\eta_{i_1'}\eta_{i_2'}\delta_{j_1'j_2'}+\eta_{i_1'}\eta_{j_2'}\delta_{i_1'j_2'}+\eta_{i_2'}\eta_{j_1'}\delta_{i_1'j_2'}+\eta_{j_1'}\eta_{j_2'}\delta_{i_1'i_2'}) \]
and $\eta\otimes_s\iota_\eta$ at the $(4,4)$-block has $(i_1',j_1',k_1',i_2',j_2',k_2')$-entry
\[\begin{split}
\frac{1}{18}\Big(&\eta_{i_1'}\eta_{i_2'}(\delta_{j_1'j_2'}\delta_{k_1'k_2'}+\delta_{j_1'k_2'}\delta_{k_1'j_2'})+\eta_{i_1'}\eta_{j_2'}(\delta_{j_1'i_2'}\delta_{k_1'k_2'}+\delta_{j_1'k_2'}\delta_{i_1'k_2'})
+\eta_{i_1'}\eta_{k_2'}(\delta_{j_1'j_2'}\delta_{k_1'i_2'}+\delta_{j_1'i_2'}\delta_{k_1'j_2'})\\&+\eta_{j_1'}\eta_{k_2'}(\delta_{i_1'j_2'}\delta_{k_1'i_2'}+\delta_{i_1'i_2'}\delta_{k_1'j_2'})
+\eta_{j_1'}\eta_{j_2'}(\delta_{i_1'i_2'}\delta_{k_1'k_2'}+\delta_{i_1'k_2'}\delta_{i_1'k_2'})+\eta_{k_1'}\eta_{k_2'}(\delta_{i_1'i_2'}\delta_{j_1'j_2'}+\delta_{i_1'j_2'}\delta_{j_1'i_2'})\\
&+\eta_{j_1'}\eta_{i_2'}(\delta_{i_1'j_2'}\delta_{k_1'k_2'}+\delta_{i_1'k_2'}\delta_{j_1'k_2'})+\eta_{k_1'}\eta_{i_2'}(\delta_{j_1'j_2'}\delta_{i_1'k_2'}+\delta_{j_1'k_2'}\delta_{i_1'j_2'})
+\eta_{j_1'}\eta_{j_2'}(\delta_{i_1'k_2'}\delta_{j_1'i_2'}+\delta_{i_1'i_2'}\delta_{j_1'k_2'})\Big) .
\end{split}
\]
We note that the principal symbol of
$\delta^s_\mathsf{F}\mathrm{d}^s_\mathsf{F}$ is the same as the one of
$\frac{1}{4}\nabla_\mathsf{F}^*\nabla_\mathsf{F}+\frac{3}{4}\mathrm{d}^s_\mathsf{F}\delta^s_\mathsf{F}$,
which is positive definite with a lower bound
$\frac{1}{4}(\xi^2+\mathsf{F}^2+|\eta|^2)$. Suppose
$a^\flat,b^\flat,c^\flat,d^\flat,d^\flat$ have a common bound $C$, then $A$
has a bound $C^2+C|\eta|+C\mathsf{F}+C\xi\leq
C'(1+\epsilon^{-1})+\epsilon(\xi^2+\mathsf{F}^2+|\eta|^2)$. Then we
can choose $\mathsf{F}>0$ large enough, and complete the proof.
\end{proof}

Let $\Omega_j$ be a domain in $M$ with boundary $\partial \Omega_j$
transversal to $\partial X$. Let $\dot{H}^{m,l}_{sc}(\Omega_j)$ be the
subspace of $H^{m,l}_{sc}(X)$ consisting of distributions supported in
$\overline{\Omega_j}$, and let $\bar{H}^{m,l}_{sc}(\Omega_j)$ be the
space of restrictions of elements of $H^{m,l}_{sc}(X)$ to
$\Omega_j$. Thus,
$\dot{H}_{sc}^{m,l}(\Omega_j)^*=\bar{H}^{-m,-l}_{sc}(\Omega_j)$.

\begin{lemma}
There exists $\mathsf{F}_0>0$ such that for $\mathsf{F}\geq\mathsf{F}_0$, the operator $\Delta_{\mathsf{F},s}=\delta_\mathsf{F}^s\mathrm{d}^s_\mathsf{F}$, considered as a map $\dot{H}^{1,0}_{sc}\rightarrow (\dot{H}^{1,0}_{sc})^*=\bar{H}_{sc}^{-1,0}$, is invertible.
\end{lemma}
\begin{proof}
Since $\delta^s_\mathsf{F}$ is defined as the adjoint of $\mathrm{d}^s_\mathsf{F}$ relative to the scattering metric, we have
\begin{equation}\label{est1}
\begin{split}
\|\mathrm{d}^s_\mathsf{F}u\|^2_{L^2_{sc}}&=\langle \mathrm{d}^s_\mathsf{F}u,\mathrm{d}^s_\mathsf{F}u\rangle=\langle\Delta_{\mathsf{F},s}u,u\rangle\\
&\leq \|\Delta_{\mathsf{F},s}u\|_{\bar{H}_{sc}^{-1,0}}\|u\|_{\dot{H}^{1,0}_{sc}}\leq \epsilon^{-1} \|\Delta_{\mathsf{F},s}u\|_{\bar{H}_{sc}^{-1,0}}^2+\epsilon \|u\|_{\dot{H}^{1,0}_{sc}}^2.
\end{split}
\end{equation}
By (\ref{fact1}) and (\ref{FFsymbol}), we have
\begin{equation}\label{fact2}
\delta^s_\mathsf{F}\mathrm{d}^s_\mathsf{F}=\frac{1}{4}\nabla^*\nabla+\frac{1}{4}\mathsf{F}^2+\frac{3}{4}\mathrm{d}^s_\mathsf{F}\delta^s_\mathsf{F}+A+\tilde{R},
\end{equation}
where $A\in\mathrm{Diff}_{sc}^1(X)$ with
\[|\langle Au, u\rangle|\leq C\|u\|_{\dot{H}^{1,0}_{sc}}\|u\|_{L^2_{sc}}+C\mathsf{F}\|u\|_{L^2_{sc}}^2,\]
and $\tilde{R}\in x\mathrm{Diff}^1_{sc}(X)$. This follows by rewriting $\nabla^*_\mathsf{F}\nabla_\mathsf{F}$ using (\ref{FFsymbol}), which modifies $R$ in $(\ref{fact1})$. Thus, we have
\[\|\mathrm{d}^s_\mathsf{F}u\|^2_{L^2_{sc}}=\frac{1}{4}\|\nabla u\|^2_{L^2_{sc}}+\frac{1}{4}\mathsf{F}^2\|u\|^2_{L^2_{sc}}+\frac{3}{4}\|\delta^s_\mathsf{F}u\|^2_{L^2_{sc}}+\langle Au,u\rangle+\langle \tilde{R}u,u\rangle.\]
This gives us
\begin{equation}
\|\nabla u\|^2_{L^2_{sc}}+\mathsf{F}^2\|u\|^2_{L^2_{sc}}\leq C\|\mathrm{d}^s_\mathsf{F}u\|^2_{L^2_{sc}}+C\|x^{1/2}u\|^2_{L^2_{sc}}+C\|u\|_{\dot{H}^{1,0}_{sc}}\|u\|_{L^2_{sc}}+C\mathsf{F}\|u\|_{L^2_{sc}}^2.
\end{equation}
Then for sufficiently large $\mathsf{F}$,
\begin{equation}
\|\nabla u\|^2_{L^2_{sc}}+\mathsf{F}^2\|u\|^2_{L^2_{sc}}\leq C\|\mathrm{d}^s_\mathsf{F}u\|^2_{L^2_{sc}}+C\|x^{1/2}u\|^2_{L^2_{sc}},
\end{equation}
where $C$ is a constant depending on $\mathsf{F}$, and thus
\[\|\nabla u\|^2_{L^2_{sc}}+\langle (1-Cx)u,u\rangle\leq C\|\mathrm{d}^s_\mathsf{F}u\|^2_{L^2_{sc}}.\]
Now suppose that $\Omega_j$ is contained in $\{x\leq x_0\}$. If $x_0$ is sufficiently small, this gives
\begin{equation}\label{Poincare}
\|\nabla u\|_{L^2_{sc}}+\|u\|_{L^2_{sc}}\leq C\|\mathrm{d}^s_\mathsf{F}u\|_{L^2_{sc}}.
\end{equation} 
If $x_0$ is larger, we can still have
\[
\|\nabla u\|_{L^2_{sc}}+\|u\|_{L^2_{sc}}\leq C\|\mathrm{d}^s_\mathsf{F}u\|_{L^2_{sc}}+C\|u\|_{L^2_{sc}(\{x_1\leq x\leq x_0\})},
\]
with $x_1$ small, and thus have $(\ref{Poincare})$ by the standard
Poincar\'{e} inequality (See \cite[Equation (28)]{SU1} for one
forms). Then, with $(\ref{est1})$, and choosing $\epsilon>0$ small, we
find that
\[
\|u\|_{\dot{H}^{1,0}_{sc}}\leq C\|\Delta_{\mathsf{F},s}u\|_{\bar{H}_{sc}^{-1,0}}.
\]
Therefore, we have proved the invertibility of $\Delta_{\mathsf{F},s}$.
 \end{proof}

Using Lemma 4.4 in \cite{SUV2}, in parallel to the above lemmas, we obtain

\begin{lemma}
There exists $\mathsf{F}_0>0$ such that for $\mathsf{F}>\mathsf{F}_0$, the operator $\Delta_{\mathsf{F},s}=\delta_\mathsf{F}^s\mathrm{d}^s_\mathsf{F}$ on symmetric $3$-tensors is invertible as a map $\dot{H}_{sc}^{1,r}\rightarrow\bar{H}_{sc}^{-1,r}$ for all $r\in\mathbb{R}$.
\end{lemma}

\begin{lemma}
Let $\Omega_j$ be a domain contained in $X$ as above. For $\mathsf{F}>0$ and $r\in\mathbb{R}$,
\[
\|u\|_{\bar{H}_{sc}^{1,r}(\Omega_j)}\leq C(\|x^{-r}\mathrm{d}^s_\mathsf{F}u\|_{L^2_{sc}(\Omega_j)}+\|u\|_{x^{-r}L^2_{sc}(\Omega_j)}),
\]
for symmetric $3$-tensors $u\in\bar{H}_{sc}^{1,r}(\Omega_j)$.
\end{lemma}

\begin{proof}
By the proof of Lemma 4.5 in \cite{SUV2}, we only need to consider the
case $r=0$. Let $\tilde{\Omega}_j$ be a domain in $X$ with $C^\infty$
boundary, transversal to $\partial X$, containing
$\overline{\Omega_j}$. We show that there exist a continuous
extension map $E:\bar{H}^{1,2}_{sc}(\Omega_j)\rightarrow
\dot{H}^{1,2}_{sc}(\tilde{\Omega}_j)$ such that
\begin{equation}\label{extension}
\|\mathrm{d}^s_\mathsf{F}Eu\|_{L^2_{sc}(\tilde{\Omega}_j)}+\|Eu\|_{L^2_{sc}(\tilde{\Omega}_j)}\leq C(\|\mathrm{d}^s_\mathsf{F}u\|_{L^2_{sc}(\Omega_j)}+\|u\|_{L^2_{sc}(\Omega_j)}),~~u\in \bar{H}^{1,0}_{sc}(\Omega_j).
\end{equation}
Once (\ref{extension}) is proved, by Lemma \ref{dd1}, with $v=Eu$, we have
\[
\begin{split}
\|\nabla v\|_{L^2_{sc}(\tilde{\Omega}_j)}^2+\|v\|_{L^2_{sc}(\tilde{\Omega}_j)}^2&\leq C(\|\mathrm{d}^s_\mathsf{F}v\|^2_{L^2_{sc}(\tilde{\Omega}_j)}+\|v\|^2_{L^2_{sc}(\tilde{\Omega}_j)})\\
& \leq C(\|\mathrm{d}^s_\mathsf{F}u\|^2_{L^2_{sc}(\Omega_j)}+\|v\|^2_{L^2_{sc}(\Omega_j)}).
\end{split}
\]
This finally gives
\[
\|u\|_{\bar{H}^{1,0}_{sc}(\Omega_j)}\leq C(\|\mathrm{d}^s_\mathsf{F}u\|^2_{L^2_{sc}(\Omega_j)}+\|v\|^2_{L^2_{sc}(\Omega_j)}).
\]

The only thing remaining is to construct $E$. By a partition of unity,
this can be reduced to the situation where locally
$X=\overline{\mathbb{R}^n}$, $\overline{\Omega_j} =
\overline{\mathbb{R}^n_+}$; see the proof of Lemma 4.5 in
\cite{SUV2}. We only need to analyze the extension of a symmetric
3-tensor on $\overline{\mathbb{R}^n_+}$ to $\overline{\mathbb{R}^n}$.

We let $\Phi_q(x',x_n')=(x',-qx_n)$ for $x_n<0$ be a diffeomorphism from $\{x_n<0\}$ to $\{x_n>0\}$. For $f_{ijk}\mathrm{d}x^{i}\otimes\mathrm{d}x^{j}\otimes\mathrm{d}x^{k}$ on $\{x_0\geq 0\}$, we define $E_1$ to be the extension to $\mathrm{R}^n$,
\[
E_1(f_{ijk}\mathrm{d}x^{i}\otimes\mathrm{d}x^{j}\otimes\mathrm{d}x^{k})(x',x_n)=\sum_{q=1}^5C_q\Phi_q^*(f_{ijk}\mathrm{d}x^{i}\otimes\mathrm{d}x^{j}\otimes\mathrm{d}x^{k}),~~x_n<0
\]
and 
\[E_1(f_{ijk}\mathrm{d}x^{i}\otimes\mathrm{d}x^{j}\otimes\mathrm{d}x^{k})(x',x_n)=f_{ijk}\mathrm{d}x^{i}\otimes\mathrm{d}x^{j}\otimes\mathrm{d}x^{k},~~x_n\geq 0,\]
with $C_q$ chosen so that $E_1:C^1(\overline{\mathbb{R}^n_+})\rightarrow C^1(\overline{\mathbb{R}^n})$. By calculation
\[
\begin{split}
&\Phi_q^*f_{ijk}\mathrm{d}x^{i}\otimes\mathrm{d}x^{j}\otimes\mathrm{d}x^{k}=f_{ijk}(x',-qx_n)\mathrm{d}x^{i}\otimes\mathrm{d}x^{j}\otimes\mathrm{d}x^{k},~~i,j,k\neq n,\\
&\Phi_q^*f_{ijn}\mathrm{d}x^{i}\otimes\mathrm{d}x^{j}\otimes\mathrm{d}x^{n}=-qf_{ijn}(x',-qx_n)\mathrm{d}x^{i}\otimes\mathrm{d}x^{j}\otimes\mathrm{d}x^{n},~~i,j\neq n,\\
&\Phi_q^*f_{inn}\mathrm{d}x^{i}\otimes\mathrm{d}x^{n}\otimes\mathrm{d}x^{n}=q^2f_{inn}(x',-qx_n)\mathrm{d}x^{i}\otimes\mathrm{d}x^{n}\otimes\mathrm{d}x^{n},~~i\neq n,\\
&\Phi_q^*f_{nnn}\mathrm{d}x^{n}\otimes\mathrm{d}x^{n}\otimes\mathrm{d}x^{n}=-q^3f_{nnn}(x',-qx_n)\mathrm{d}x^{n}\otimes\mathrm{d}x^{n}\otimes\mathrm{d}x^{n},\\
&\partial_l\Phi_q^*f_{ijk}\mathrm{d}x^{i}\otimes\mathrm{d}x^{j}\otimes\mathrm{d}x^{k}=\partial_lf_{ijk}(x',-qx_n)\mathrm{d}x^{i}\otimes\mathrm{d}x^{j}\otimes\mathrm{d}x^{k},~~i,j,k,l\neq n,\\
&\partial_l\Phi_q^*f_{ijn}\mathrm{d}x^{i}\otimes\mathrm{d}x^{j}\otimes\mathrm{d}x^{n}=-q\partial_lf_{ijn}(x',-qx_n)\mathrm{d}x^{i}\otimes\mathrm{d}x^{j}\otimes\mathrm{d}x^{n},~~i,j,l\neq n,\\
&\partial_l\Phi_q^*f_{inn}\mathrm{d}x^{i}\otimes\mathrm{d}x^{n}\otimes\mathrm{d}x^{n}=q^2\partial_lf_{inn}(x',-qx_n)\mathrm{d}x^{i}\otimes\mathrm{d}x^{n}\otimes\mathrm{d}x^{n},~~i,l\neq n,\\
&\partial_l\Phi_q^*f_{nnn}\mathrm{d}x^{n}\otimes\mathrm{d}x^{n}\otimes\mathrm{d}x^{n}=-q^3\partial_lf_{nnn}(x',-qx_n)\mathrm{d}x^{n}\otimes\mathrm{d}x^{n}\otimes\mathrm{d}x^{n},~~l\neq 0\\
&\partial_n\Phi_q^*f_{ijk}\mathrm{d}x^{i}\otimes\mathrm{d}x^{j}\otimes\mathrm{d}x^{k}=-q\partial_nf_{ijk}(x',-qx_n)\mathrm{d}x^{i}\otimes\mathrm{d}x^{j}\otimes\mathrm{d}x^{k},~~i,j,k\neq n,\\
&\partial_n\Phi_q^*f_{ijn}\mathrm{d}x^{i}\otimes\mathrm{d}x^{j}\otimes\mathrm{d}x^{n}=q^2\partial_nf_{ijn}(x',-qx_n)\mathrm{d}x^{i}\otimes\mathrm{d}x^{j}\otimes\mathrm{d}x^{n},~~i,j\neq n,\\
&\partial_n\Phi_q^*f_{inn}\mathrm{d}x^{i}\otimes\mathrm{d}x^{n}\otimes\mathrm{d}x^{n}=-q^3\partial_nf_{inn}(x',-qx_n)\mathrm{d}x^{i}\otimes\mathrm{d}x^{n}\otimes\mathrm{d}x^{n},~~i\neq n,\\
&\partial_n\Phi_q^*f_{nnn}\mathrm{d}x^{n}\otimes\mathrm{d}x^{n}\otimes\mathrm{d}x^{n}=q^4\partial_nf_{nnn}(x',-qx_n)\mathrm{d}x^{n}\otimes\mathrm{d}x^{n}\otimes\mathrm{d}x^{n}.
\end{split}
\]
The matching of the derivatives at $x_n=0$, which gives the $C^1$ property, yields
\[
\begin{split}
&C_1+C_2+C_3+C_4+C_5=1,\\
&C_1+2C_2+3C_3+4C_4+5C_5=-1,\\
&C_1+4C_2+9C_3+16C_4+25C_5=1,\\
&C_1+8C_2+27C_3+64C_4+125C_5=-1,\\
&C_1+16C_2+81C_3+256C_4+625C_5=1.
\end{split}
\]
The linear system, with a Vandermonde matrix, is solvable. With the
$C_q,q=1,2,\cdots,5$, satisfying the linear system above, we obtain the
property $E_1:C^1_c(\overline{\mathbb{R}^n_+})\rightarrow
C^1_c(\overline{\mathbb{R}^n})$ and
\[
\|E_1u\|_{H^1(\mathbb{R}^n)}\leq C\|u\|_{H^1(\mathbb{R}^n_+)}.
\]

With $\Phi^*_q$ acting on 4-tensors as usual, we have
\[
\mathrm{d}^s\Phi^*_q=\Phi_q^*\mathrm{d}^s,
\]
and thus
\[
\|\mathrm{d}^s\Phi^*_qu\|_{L^2(\mathbb{R}^n_-)}\leq C\|\mathrm{d}^su\|_{L^2(\mathbb{R}^n_+)}.
\]
Then
\[
\|\mathrm{d}^sE_1u\|_{L^2(\mathbb{R}^n)}\leq C\|\mathrm{d}^su\|_{L^2(\mathbb{R}_+^n)}
\]
which completes the proof.
\end{proof}

Now we define
\begin{eqnarray*}
&\mathcal{S}_{\mathsf{F},\Omega_j}=\mathrm{Id}-\mathrm{d}^s_{\mathsf{F}}\Delta_{\mathsf{F},s}^{-1}\delta^s_\mathsf{F},\\
&\mathcal{P}_{\mathsf{F},\Omega_j}=\mathrm{d}^s_{\mathsf{F}}Q_{\mathsf{F},\Omega_j},~~Q_{\mathsf{F},\Omega_j}=\Delta_{\mathsf{F},s}^{-1}\delta^s_\mathsf{F}.
\end{eqnarray*}
In parallel to Corollaries 4.6, 4.7, 4.8 in \cite{SUV2}, we have
obtain for the Dirichlet Laplacian $\Delta_{\mathsf{F},s}$

\begin{corollary}\label{coro46}
Let $\phi$ on $C_c^\infty(\overline{\Omega_j}\setminus\partial_{\mathrm{int}}\Omega_j)$. Then on symmetric $3$-tensors, there exists $\mathsf{F}_0>0$ such that for any $\mathsf{F}\geq\mathsf{F}_0$, $\phi\Delta_{\mathsf{F},s}^{-1}\phi:\bar{H}_{sc}^{-1,k}\rightarrow \dot{H}_{sc}^{1,k}$ is in $\Psi_{sc}^{-2,0}(X)$.
\end{corollary}
%\begin{proof}
%By Lemma \ref{dd1}, $\Delta_{\mathsf{F},s}$ has a parametrix $B\in\Psi_{sc}^{-2,0}(X)$ so that
%\[
%B\Delta_{\mathsf{F},s}=\mathrm{Id}+F_L,~~\Delta_{\mathsf{F},s}B=\mathrm{Id}+F_R,
%\]
%with $F_{L},F_R\in \Psi_{sc}^{-\infty,-\infty}(X)$. Let $\psi\in C^\infty_c(\overline{\Omega}\setminus\partial_{\mathrm{int}}\Omega_j)$ be identically $1$ on $\mathrm{supp}\phi$. Then
%\[
%\psi=B\Delta_{\mathsf{F},s}\psi-F_L\psi= B\psi\Delta_{\mathsf{F},s}
%\]
%\end{proof}
~\\
\begin{corollary}\label{coro47}
Let $\phi\in C_c^\infty(\overline{\Omega_j}\setminus\partial_{\mathrm{int}}\Omega_j)$, $\chi\in C^\infty(\overline{\Omega_j})$ with disjoint support and with $\chi$ constant near $\partial_{\mathrm{int}}\Omega$. Let $\mathsf{F},\mathsf{F}_0$ as in Corollary \textnormal{\ref{coro46}}. Then the operator $\chi\Delta_{\mathsf{F},s}^{-1}\phi:\bar{H}_{sc}^{-1,k}(\Omega_j)\rightarrow\dot{H}^{1,k}_{sc}(\Omega_j)$ in fact maps $H^{s,r}_{sc}(X)\rightarrow\dot{H}_{sc}^{1,k}(\Omega_j)$ for all $s,r,k$.

Similarly, $\phi\Delta_{\mathsf{F},s}^{-1}\chi:\bar{H}_{sc}^{-1,k}(\Omega_j)\rightarrow\dot{H}^{1,k}_{sc}(\Omega_j)$ in fact maps $\bar{H}_{sc}^{-1,k}(\Omega_j)\rightarrow H^{s,r}_{sc}(X)$ for all $s,r,k$.
\end{corollary}
~\\
\begin{corollary}\label{coro48}
Let $\phi\in C_c^\infty(\overline{\Omega_j}\setminus \partial_{\mathrm{int}}\Omega_j)$, $\chi\in C^\infty(\overline{\Omega_j})$ with disjoint support and with $\chi$ constant near $\partial_{\mathrm{int}}\Omega_j$. Let $\mathsf{F},\mathsf{F}_0$ as in Corollary \textnormal{\ref{coro46}}.

Then $\phi\mathcal{S}_{\mathsf{F},\Omega_j}\phi\in\Psi_{sc}^{0,0}(X)$, while $\chi\mathcal{S}_{\mathsf{F},\Omega_j}\phi:H^{s,r}_{sc}(X)\rightarrow x^k L^2_{sc}(\Omega_j)$ and $\phi\mathcal{S}_{\mathsf{F},\Omega_j}\chi:x^k L^2_{sc}(\Omega_j)\rightarrow H^{s,r}_{sc}(X)$ for all $s,r,k$.
\end{corollary}
~\\

We also have properties in parallel to Lemmas 4.9 to 4.13, and then
arrive at the main, local, result

\begin{theorem}
For $\Omega=\Omega_c$, $c>0$ small, there exists $\mathsf{F}_0>0$
large enough, such that for $\mathsf{F}>\mathsf{F}_0$, the geodesic
ray transform on symmetric $4$-tensors $f\in
e^{\mathsf{F}/x}L_{sc}^2(\Omega)$ satisfying
$\delta^s(e^{-2\mathsf{F}/x}f)=0$, is injective.% with a stability estimates. Here the stability is in the sense
%that for $s\geq 0$ there exist $R,R'$ such that for any sufficiently negative $r$ that
%the $e^{\mathsf{F}/x}H_{sc}^{s-1,r}$ norm of $f$ on $\Omega$ is controlled by the
%$e^{\mathsf{F}/x}H_{sc}^{s,r+R}$ norm of $I_4f$, provide that $f$ is a priori in $e^{\mathsf{F}/x}H_{sc}^{s,r+R'}$.
\end{theorem}
 
The above local theorem leads to the global result, Theorem
\ref{mainth} similar to \cite[Theorem 4.19]{SUV2}.

\section*{Acknowledgments}

M. V. de Hoop gratefully acknowledges support from the Simons
Foundation under the MATH + X program, the National Science Foundation
under grant DMS-1815143, and the corporate members of the
Geo-Mathematical Group at Rice University. G. Uhlmann was partly supported by NSF, a Walker Family Endowed Professorship at UW and a Si-Yuan Professorship at IAS, HKUST. J. Zhai acknowledges the great hospitality of IAS, HKUST, where this work was initiated.

{}

\begin{thebibliography}{99}
\bibitem{AR} Yu. Anikonov, V. Romanov, \emph{On uniqueness of determination of a form of first degree by it integrals along geodesics}, J. Inverse Ill-Posed Probl. \text{5} (1997), 467-480.
\bibitem{CJ} V. \u{C}erveny, J. Jech, \emph{Linearized solutions of kinemetic problems of seismic body waves in inhomogeneous slightly anisotropic media}, J. Geophys., \textbf{51} (1982), 96-104.
\bibitem{Dair} N. S. Dairbekov, \emph{Integral geometry problem for nontrapping manifolds}, Inverse Problems \textbf{22} (2006), 431-445.
\bibitem{Her} G. Herglotz, \emph{\"{U}ber die Elastizitaet der Erde bei Beruecksichtgung ihrer variablen Dichte}, Zeitschr. f\"{u}r Math. Phys., \textbf{52} (1905), 275-299
\bibitem{IM} J. Ilmavirta, F. Monard, \emph{Integral geometry on manifolds with boundary and applications}, preprint, arXiv:1806.06088
\bibitem{Melrose} R. Melrose, \emph{Geometric scattering theory}, Cambridge, 1995
\bibitem{Muk} R. G. Mukhometov, \emph{The reconstruction problem of a two-dimensional Riemannian metric, and integral geometry}, Dokl. Akad. Nauk SSSR \textbf{232} (1977), 32-35.
\bibitem{Muk1} R. G. Mukhometov, \emph{On a problem of reconstructing Riemannian metrics}, Sibirsk. Mat. Zh., \textbf{22} (3), (1987) 119-135, 237.
\bibitem{Muk2} R. G. Mukhometov, \emph{On the problem of integral geometry (Russian)}, Math. problems of geophysics. Akad. Nauk SSSR, Sibirsk., Otdel., Vychisl., Tsentr, Novosibirsk, 6 (2) (1975), 212-242.
\bibitem{PSU} G. Paternain, M. Salo, G. Uhlmann, \emph{Tensor tomography on simple surfaces}, Invent. Math., \textbf{193} (2013), 229-247.
\bibitem{PSUZ} G. Paternain, M. Salo, G. Uhlmann, H. Zhou, \emph{The geodesic X-ray transform with matrix weights}, preprint.
\bibitem{Pestov} L. Pestov, \emph{Well-posedness questions of the ray tomography problems}, (Russian), Siberian Science Press, Novosibirsk, 2003.
\bibitem{PS} L. Pestov, V. A. Sharafutdinov, \emph{Integral geometry of tensor fields on a manifold of negative curvature}, Siberian Math. J., \textbf{29} (1988), 427-441.
\bibitem{Shara} V. A. Sharafutdinov, \emph{Integral geometry of tensor fields}, Inverse and Ill-Posed Problems Series. VSP, Utrecht, 1994.
\bibitem{SU1} P. Stefanov, G. Uhlmann, \emph{Stability estimates for the X-ray transform of tensor fields and boundary rigidity}, Duke Math. J., \textbf{123} (2004), 445-467
\bibitem{SU} P. Stefanov, G. Uhlmann, \emph{Boundary ridigity and stability for generic simple metrics}, Journal of Amer. Math. Soc., \textbf{18} (2005), 975-1003.
\bibitem{SUV1} P. Stefanov, G. Uhlmann, A. Vasy, \emph{Boundary rigidity with partial data}, J. Amer. Math. Soc., \textbf{29}, pp. 299-332 (2015).
\bibitem{SUV2} P. Stefanov, G. Uhlmann, A. Vasy, \emph{Inverting the local geodesic X-ray transform on tensors}, arXiv:1410.5145v1.
\bibitem{SUVZ} P. Stefanov, G. Uhlmann, A. Vasy, H. Zhou, \emph{Travel time tomography}, preprint
\bibitem{UV} G. Uhlmann, A. Vasy, \emph{The inverse problem for the local geodesic ray transform}, Invent. Math., \textbf{205}, pp. 83-120 (2016). 
\bibitem{WZ} E. Wiechert, K. Zoeppritz, \emph{\"{U}ber Erdbebenwellen}, Nachr. Koenigl. Geselschaft Wiss. G\"{o}ttingen, \textbf{4} (1907), 415-549
\end{thebibliography}
\end{document}